\documentclass[12pt]{amsart}
\usepackage[headings]{fullpage}
\usepackage{amssymb,epic,eepic,epsfig,amsbsy,amsmath,amscd,bm,overpic}
\usepackage{url}
\usepackage[all]{xy}
\numberwithin{equation}{section}
                        \theoremstyle{plain}
\usepackage{mathrsfs}

\newcommand\no[1]{}

\newtheorem{theorem}{Theorem}[section]
\newtheorem{thm}{Theorem}
\newtheorem{lemma}[theorem]{Lemma}

\newtheorem{proposition}[theorem]{Proposition}

\theoremstyle{definition}
\newtheorem{remark}[theorem]{Remark}

\def\BC{\mathbb C}

\def\BZ{\mathbb Z}
\def\BR{\mathbb R}

\def\BQ{\mathbb Q}

\def\CK{\mathcal K}

\def\la{\langle}
\def\ra{\rangle}

\DeclareMathOperator{\tr}{\mathrm tr}

\def\ve{\varepsilon}
\def\be { \begin{equation} }
\def\ee { \end{equation} }

\begin{document}

\title[$\mathrm{SL}_2(\BR)$-representations and left-orderable surgeries]{$\mathrm{SL}_2(\BR)$-representations and left-orderable surgeries of $(-2, 3, 2n+1)$-pretzel knots}

\author{Anh T. Tran}


\begin{abstract}
In this paper, we provide an explicit construction of continuous paths of $\mathrm{SL}_2(\BR)$-representations of the knot groups of $(-2,3,2n+1)$-pretzel knots. As an application,
we show that the fundamental group of the $3$-manifold obtained from the $3$-sphere by $\frac{m}{l}$-surgery along the $(-2,3,2n+1)$-pretzel knot, where $n \ge 3$ is an integer and $n \not= 4$, is left-orderable if $\frac{m}{l}< 2 \lfloor \frac{2n+4}{3} \rfloor$. 
\end{abstract}

\address{Department of Mathematical Sciences, The University of Texas at Dallas, 
Richardson, TX 75080, USA}
\email{att140830@utdallas.edu}

\maketitle

\section{Introduction}

For a rational homology $3$-sphere $Y$, the rank of the hat version of its Heegaard Floer homology is bounded below by the order of its first integral homology group \cite{OS}. It is called an L-space if equality holds, namely $\mathrm{rk}(\widehat{\mathrm{HF}}(Y)) = |H_1(Y; \BZ)|$. Examples of L-spaces include lens spaces and all manifolds with finite fundamental groups. It is natural to ask if there are non-Heegaard Floer characterizations of L-spaces. The L-space conjecture of Boyer, Gordon and Watson \cite{BGW} is an attempt to answer this question. It states that an irreducible rational homology 3-sphere is an L-space if and only if its fundamental group is not left-orderable. Here, a non-trivial group $G$ is called left-orderable if it admits a total ordering $<$ such that $g<h$ implies $fg < fh$ for all elements $f,g,h$ in $G$. 

We now focus on $3$-manifolds obtained from $S^3$ by Dehn surgeries along a knot. For double twist knots and classical odd pretzel knots, intervals of slopes for which the 3-manifold obtained from $S^3$ by Dehn surgery along the knot has left-orderable fundamental group was determined in \cite{BGW, HT-52, HT-genus1, Tr, Tr-I, Ga, KTT, KT, KL, Wa}, by using continuous paths of $\mathrm{SL}_2(\BR)$-representations of the knot group and the fact that the universal covering group of $\mathrm{SL}_2(\BR)$ is a left-orderable group. 

An algebraic approach to the understanding of left-orderability of the fundamental groups of $3$-manifolds obtained by Dehn surgeries along a knot in $S^3$ was proposed by Culler and Dunfield \cite{CD}, by using elliptic $\mathrm{SL}_2(\BR)$-representations. See also \cite{HZ}. Similar approaches using hyperbolic $\mathrm{SL}_2(\BR)$-representations and detected Seifert surfaces were also proposed by Gao \cite{Ga-holonomy} and Wang \cite{Wa} respectively. See  \cite{Zu, Hu-taut} and references therein for a geometric approach using taut foliations. 

By applying the $\mathrm{SL}_2(\BR)$-representation technique, Nie \cite{Ni} proved that the $3$-manifold obtained by Dehn surgery along the $(-2,3,2n+1)$-pretzel knot, where $n \ge 3$ is an integer, has left-orderable fundamental group if the surgery slope is in a small neighborhood of $0$.  Moreover, Varvarezos \cite{Va} proved that the $3$-manifold obtained by $\frac{m}{l}$-surgery along the $(-2,3,7)$-pretzel knot has left-orderable fundamental group if $\frac{m}{l}< 6$. 

In this paper, we provide an explicit construction of continuous paths of $\mathrm{SL}_2(\BR)$-representations and a detailed longitude calculation of the knot groups of $(-2,3,2n+1)$-pretzel knots, by using the  descriptions of their $\mathrm{SL}_2(\BC)$-character varieties obtained in previous works \cite{LT, Tr-ring}. The precise construction and calculation are described in details in Section \ref{path} and Section \ref{long} respectively. As an application, we study left-orderable surgeries of $(-2,3,2n+1)$-pretzel knots and show the following.

\begin{thm} \label{main}
Let $n \ge 3$ is an integer and $n \not= 4$. Then the $3$-manifold obtained from $S^3$ by $\frac{m}{l}$-surgery along the $(-2,3,2n+1)$-pretzel knot has left-orderable fundamental group if $\frac{m}{l}< 2 \lfloor \frac{2n+4}{3} \rfloor$. 
\end{thm}

It is known that $(-2,3,2n+1)$-pretzel knots, with $n \ge 3$, are L-space knots, see e.g. \cite{BM}. Here, a knot $K$ in $S^3$ is called an L-space knot if it admits a positive Dehn surgery yielding an $L$-space. For an L-space knot $\CK$, it is known that the $3$-manifold obtained by $\frac{m}{l}$-surgery along $\CK$ is an L-space if and only if $\frac{m}{l}\ge 2g(\CK)-1$, where $g(\CK)$ is the Seifert genus of $\CK$, see \cite{OS}. From the view point of the L-space conjecture, we would expect that the $3$-manifold obtained by $\frac{m}{l}$-surgery along $\CK$ has left-orderable fundamental group if and only if $\frac{m}{l} < 2g(\CK)-1$. For the $(-2,3,2n+1)$-pretzel knot, with $n \ge 3$, the genus is $n+2$ and therefore the L-space conjecture predicts that any slope in the interval $(-\infty, 2n+3)$ yields a $3$-manifold with left-orderable fundamental group by Dehn surgery. Theorem \ref{main} confirms this prediction for the interval $(-\infty, 2 \lfloor \frac{2n+4}{3} \rfloor)$. The  interval $[2 \lfloor \frac{2n+4}{3} \rfloor, 2n+3)$ is unsolved via the $\mathrm{SL}_2(\BR)$-representation technique; however, it has been recently solved by Zhao \cite{Zh} via the taut foliation technique. 


\no{
\begin{figure}[h]
	\centering
    \begin{overpic}[width=0.45 \textwidth,tics=9]{knotBlankDirected0.png}
        \put (12, 40) {\small{\text{$a_1$ crossings}}}
        \put (52, 40) {\small{\text{$a_2$ crossings}}}
        \put (92, 40) {\small{\text{$a_3$ crossings}}}
        \put (68, 12) {{$x_1$}}
        \put (48, 2) {{$x_2$}}
        \put (28, 12) {{$x_3$}}
    \end{overpic}
    	\caption{The pretzel knot $P(a_1, a_2, a_3)$.}
	\label{fig:pretzel}
\end{figure}
}

The paper is organized as follows. In Section \ref{rep} we review $\mathrm{SL}_2(\BC)$-representations of $(-2, 3, 2n+1)$-pretzel knots $\CK_n$ and describe one-variable parametrizations of these representations. In Section \ref{Alex}, we find certain roots of the Alexander polynomial of $\CK_n$ on the unit circle, which will then be used in Section \ref{path} to construct continuous paths of elliptic $\mathrm{SL}_2(\BR)$-representations of $\CK_n$. In Section \ref{long}, we study the image of a canonical longitude under these paths. Finally, we give a proof of Theorem \ref{main} in Section \ref{LO}.

\section{$\mathrm{SL}_2(\BC)$-representations} \label{rep}

In this section we first recall Chebyshev polynomials and their properties. Then we review $\mathrm{SL}_2(\BC)$-representations of $(-2, 3, 2n+1)$-pretzel knots from \cite{LT}. Finally, we follow \cite{Tr-ring} to find one-variable parametrizations of these representations. 

\subsection{Chebychev polynomials}

Let $\{S_k(y)\}_{k \in \BZ}$ be Chebychev polynomials of the second kind defined by $S_0(y) = 1$, $S_1(y) = y$ and $S_k(y) = y S_{k-1}(y) - S_{k-2}(y)$ for all integers $k$. By induction we have $S_k(\pm2) = (\pm 1)^k (k+1)$ and $S_k(y) = (t^{k+1} - t^{-k-1}) / (t - t^{-1})$ for $y = t+ t^{-1} \not= \pm 2$. Using this, we can prove the following properties.

\begin{lemma} \label{chev}
For $k \in \BZ$ we have
$$S^2_k(y) + S^2_{k-1}(y) - y S_{k}(y) S_{k-1}(y) = 1.$$
\end{lemma}

\begin{lemma} \label{chev-sign}
For $k \ge 1$ we have

$(1)$ $S_{k}(y) = \prod_{j=1}^{k} (y-2\cos \frac{j\pi}{k+1})$. Moreover $(-1)^{j}S_{k}(y) > 0$ if $y \in (2\cos \frac{(j+1)\pi}{k+1}, 2\cos \frac{j\pi}{k+1})$.

$(2)$ $S_{k}(y) - S_{k-1}(y) = \prod_{j=1}^{k} (y-2\cos \frac{(2j-1)\pi}{2k+1})$. Moreover $(-1)^{j} (S_{k}(y) - S_{k-1}(y) ) > 0$ if $y \in (2\cos \frac{(2j+1)\pi}{2k+1}, 2\cos \frac{(2j-1)\pi}{2k+1})$.
\end{lemma}

For any matrix $A \in \mathrm{SL}_2(\BC)$, by the Cayley-Hamilton theorem, we have $A^2 - (\tr A) A + I = O$, where $I$ and $O$ denote the $2 \times 2$ identity matrix and zero matrix respectively. Then, by induction we obtain the following. 

\begin{lemma} \label{power}
For $A \in \mathrm{SL}_2(\BC)$ and $k \in \BZ$ we have 
$$
A^k = S_{k-1}(\tr A) A - S_{k-2}(\tr A) I.
$$
\end{lemma}

\subsection{$\bm{\mathrm{SL}_2(\BC)}$-representations} For a knot $K$ in $S^3$, let $G(K)$ denote the knot group of $K$ which is the fundamental group of the complement of an open tubular neighborhood of $K$. 
For the $(-2,3,2n+1)$-pretzel knot $\CK_n$ we have
$$G(\CK_n)=\la a,b,c \mid cacb=acba,~ba(cb)^n=a(cb)^nc\ra,$$
where $a,b,c$ are meridians depicted in Figure \ref{figp}.

\begin{figure}[th]
\centerline{\psfig{file=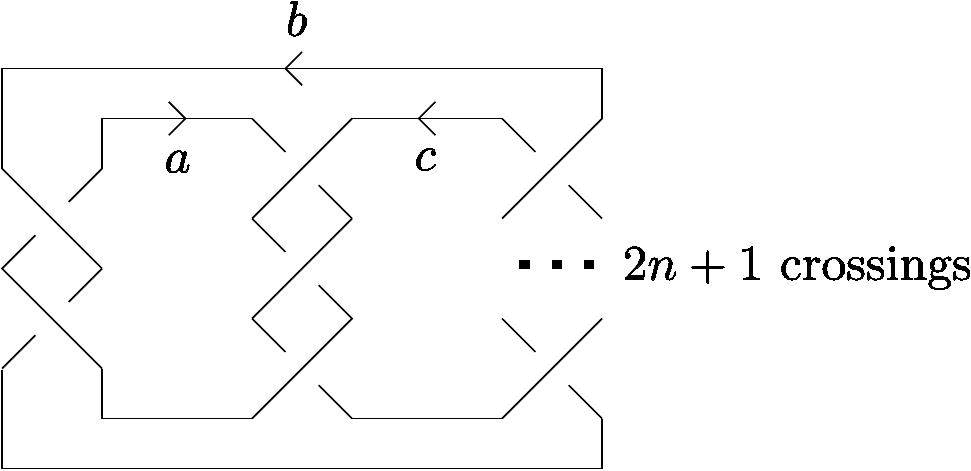,width=3.5in}}
\caption{The $(-2,3,2n+1)$-pretzel knot.}
\label{figp}
\end{figure}

According to \cite{LT}, we can find a presentation of $G(\CK_n)$ with two generators and one relator as follows. Let $w := cb$. Then the relation $cacb=acba$ becomes $caw=awa.$ This implies that $c=awaw^{-1}a^{-1}$ and so $b=c^{-1}w=awa^{-1}w^{-1}a^{-1}w$. The relation $ba(cb)^n=a(cb)^nc$ can be written as $awa^{-1}w^{-1}a^{-1}waw^n=aw^nawaw^{-1}a^{-1}$, i.e.
$$w^nawa^{-1}w^{-1}a^{-1}=a^{-1}w^{-1}awaw^{-1}w^n.$$
Hence we obtain a desired presentation 
$$G(\CK_n) = \la a,w \mid w^n u = v w^n\ra$$
where $u :=awa^{-1}w^{-1}a^{-1}$ and $v:=a^{-1}w^{-1}awaw^{-1}.$

The $\mathrm{SL}_2(\BC)$-character variety of the free group $\la a,w \ra$ is isomorphic to $\BC^3$ by the Fricke-Klein-Vogt theorem, see e.g. \cite{LM}. 
For any representation $\rho: \la a,w \ra  \to \mathrm{SL}_2(\BC)$ we let $x := \tr \rho(a)$, $y := \tr \rho(w)$ and $z := \tr \rho(aw)$. 
By \cite[Theorem 4.1]{LT}, the representation $\rho$  extends to a representation $\rho: G(\CK_n) \to \mathrm{SL}_2(\BC)$ if and only if both $ \tr \rho(u) = \tr \rho(v)$ and $\tr \rho(w^n u a) =\tr\rho(vw^n a)$ hold true. 

With
$Q(x,y,z) := \tr \rho(u) - \tr \rho(v)$ and $R_n(x,y,z) := \rho(w^n u a) - \tr\rho(vw^n a)$, by applying Lemma \ref{power} we obtain the following explicit formulas
\begin{eqnarray*}
Q &=& x - x y + (x^2+y^2+z^2-xyz-3) z,\\
R_n&=& (y+2-xz-x^2-z^2+xyz)S_{n-2}(y) - (y^2+y-2+z^2-xyz)S_{n-3}(y).
\end{eqnarray*}

\subsection{Solving the equations} We now solve $Q=R_n=0$. 
As in \cite[Section 4]{Tr-ring}, we first consider $Q$ and $R_n$ as quadratic polynomials in $x$ with coefficients in $\BC[y,z]$:
\begin{eqnarray*}
Q &=& x^2 z - x (yz^2+y-1)+ (y^2+z^2-3) z, \\
R_n &=& - x^2 S_{n-2}(y) + x z (  (y -1) S_{n-2}(y)  + y S_{n-3}(y)) \\
&& + \, (y+2-z^2)S_{n-2}(y) - (y^2+y-2 + z^2)S_{n-3}(y).
\end{eqnarray*}
This implies that $S_{n-2}(y) Q + z R_n = -p x + q z$, where
\begin{eqnarray*}
p &:=& (y-1) S_{n-2}(y) + z^2 (S_{n-2}(y)  -y S_{n-3}(y)), \\
q &:=& (y^2+y-1) S_{n-2}(y) - (y^2+y-2) S_{n-3}(y)- z^2 S_{n-3}(y).
\end{eqnarray*}

Assume $pz  \not= 0$. Then $Q=R_n=0$ is equivalent to $Q = - p x + q z = 0$, which is also equivalent to $x= z q/p$ and $S: = z^2 q^2 - (yz^2+y-1)pq + (y^2+z^2-3)p^2 =0$. 

By a direct calculation we have
\begin{eqnarray*}
S &=& (z^2 - y^3+3y-2)[ (S^2_{n-2}(y) + S^2_{n-3}(y) - y S_{n-2}(y) S_{n-3}(y)) (z^2-1)^2   \\
&& \qquad \qquad \qquad \qquad - \, S_{n-2}(y) (S_{n-2}(y) -S_{n-3}(y)) (z^2-1)- (S_{n-2}(y) -S_{n-3}(y))^2].
\end{eqnarray*}
Since $S^2_{n-2}(y) + S^2_{n-3}(y) - y S_{n-2}(y) S_{n-3}(y)=1$, we obtain $S = (z^2-y^3+3y-2) T$ where
$$
T : = (z^2-1)^2  - S_{n-2}(y) (S_{n-2}(y) -S_{n-3}(y)) (z^2 - 1)  -  (S_{n-2}(y) -S_{n-3}(y))^2. 
$$
Hence $Q=R_n=0$ if $pz \not= 0$, $x = zq/p$ and $(z^2-y^3+3y-2) T=0$.

Assume $pz \not= 0$, $x = zq/p$ and $T=0$.  By considering $T$ as a quadratic polynomial in $z^2-1$ with coefficients in $\BC[y]$ we have
$$ 
z^2 -1 = (S_{n-2}(y)  - S_{n-3}(y))  \frac{S_{n-2}(y) \pm \sqrt{4 + S^2_{n-2}(y)}}{2}. 
$$
Let $r := \frac{S_{n-2}(y) \pm \sqrt{4 + S^2_{n-2}(y)}}{2}$. Then $r^2 - S_{n-2}(y) r - 1 = 0$. Since $z^2 - 1 = (S_{n-2}(y) - S_{n-3}(y)) r$ we have   
\begin{eqnarray*}
p &=& (S_{n-2}(y) - S_{n-3}(y))[ y  + r ( S_{n-2}(y) - y S_{n-3}(y))], \\
q &=& (S_{n-2}(y) - S_{n-3}(y)) [y^2+y-1 - r S_{n-3}(y)].
\end{eqnarray*}
Hence 
$$
x = z \frac{q}{p} = z \frac{y^2+y-1 - r S_{n-3}(y)  }{  y + r ( S_{n-2}(y) - y S_{n-3}(y))}.
$$

We want to write $\frac{y^2+y-1 - r S_{n-3}(y)  }{  y + r ( S_{n-2}(y) - y S_{n-3}(y))} = g + h r$ for some rational functions $g, h$ in $y$. Since $r^2 = S_{n-2}(y) r + 1$, this is equivalent to
\begin{eqnarray*}
0 
&=&  g y - y^2 - y +1 + r [g ( S_{n-2}(y) - y S_{n-3}(y)) + h y+ S_{n-3}(y)]  \\
&& + \,  h ( S_{n-2}(y) - y S_{n-3}(y)) ( S_{n-2}(y) r + 1),
\end{eqnarray*}
which holds true if  $g,h$ are solutions of the following system
\begin{eqnarray*}
 g( S_{n-2}(y) - y S_{n-3}(y)) + h [y +( S_{n-2}(y) - y S_{n-3}(y)) S_{n-2}(y) ] &=& - S_{n-3}(y), \\
g y + h ( S_{n-2}(y) - y S_{n-3}(y)) &=& y^2+y -1.
\end{eqnarray*}

Note that since $S^2_{n-2}(y) + S^2_{n-3}(y) -y S_{n-2}(y) S_{n-3}(y) = 1$ we can write 
\begin{eqnarray*}
&& y +( S_{n-2}(y) - y S_{n-3}(y)) S_{n-2}(y) \\
&=& (y+1) S^2_{n-2}(y) + y S^2_{n-3}(y) - (y^2+y)  S_{n-2}(y) S_{n-3}(y).
\end{eqnarray*}
By direct calculations we have
\begin{eqnarray*}
D &=& \begin{vmatrix}
S_{n-2}(y) - y S_{n-3}(y) & (y+1) S^2_{n-2}(y) + y S^2_{n-3}(y) - (y^2+y)  S_{n-2}(y) S_{n-3}(y)\\
y & S_{n-2}(y) - y S_{n-3}(y)
\end{vmatrix}
\\
&=& - [(y^2+y-1) S_{n-2}(y) - (y^3 + y^2 - 2y) S_{n-3}(y)] S_{n-2}(y), \\
D_g &=&  \begin{vmatrix}
- S_{n-3}(y) & (y+1) S^2_{n-2}(y) + y S^2_{n-3}(y) - (y^2+y)  S_{n-2}(y) S_{n-3}(y)\\
y^2+y -1 & S_{n-2}(y) - y S_{n-3}(y)
\end{vmatrix}
\\
&=& - [(y^2+y -1) S_{n-2}(y) - (y^3+y^2-2y) S_{n-3}(y)] [(y+1)S_{n-2}(y) - S_{n-3}(y)], \\
D_h 
&=& \begin{vmatrix}
S_{n-2}(y) - y S_{n-3}(y) & - S_{n-3}(y) \\
y & y^2+y -1
\end{vmatrix}
\\
&=& (y^2+y -1) S_{n-2}(y) - (y^3+y^2-2y) S_{n-3}(y).
\end{eqnarray*}
Assume $S_{n-2}(y) \not=0$. We can take $g= \frac{(y+1)S_{n-2}(y) - S_{n-3}(y)}{S_{n-2}(y)}$ and $h = - \frac{1}{S_{n-2}(y)}$. Then
$$
x = z (g + hr) = z \left( y+1-\frac{S_{n-3}(y)+r}{S_{n-2}(y)} \right).
$$
Hence we have shown that $Q=R_n=0$ if $pz S_{n-2}(y) \not= 0$, $x = z \left( y+1-\frac{S_{n-3}(y)+r}{S_{n-2}(y)} \right)$ and $z^2 - 1 = (S_{n-2}(y) - S_{n-3}(y)) r$. 

However, it turns out that the condition $pz \not= 0$ can be relaxed. More precisely, we have the following. 

\begin{proposition} \label{sol}
Let $r \in \BC^*$ such that $r - r^{-1} =S_{n-2}(y)$. 
Suppose $S_{n-2}(y) \not=0$, $x =z(y+1-\frac{S_{n-3}(y)+r}{S_{n-2}(y)})$ and $z^2 = 1 + (S_{n-2}(y) - S_{n-3}(y))r$. Then $Q=R_n=0$.
\end{proposition}

\begin{proof}
We first prove that $Q=0$. Let $d= \frac{S_{n-2}(y) - S_{n-3}(y) - r}{S_{n-2}(y)}$. Since $x = z (y + d)$ we have $Q= z Q'$, where 
\begin{eqnarray*}
Q' &=& (y+d) (1-y) + x^2+y^2+z^2-xyz-3 \\
&=& (y + d)(1-y) + (y + d)^2 z^2  + y^2 + z^2 - (y + d) y z^2 -3
\\
&=& d(1-y) + y-3 + (d^2+yd+1) z^2 \\
&=& d^2+d+y-2 + (d^2+yd+1) (z^2-1) \\
&=& d^2+d+y -2+ (d^2+yd+1) (S_{n-2}(y) - S_{n-3}(y)) r.
\end{eqnarray*}

Write $d = e-f$, where $e =  \frac{S_{n-2}(y) - S_{n-3}(y)}{S_{n-2}(y)} $ and $f = \frac{r}{S_{n-2}(y)}$. Since $r^2 - S_{n-2}(y) r = 1$ we have $f^2 - f = \frac{1}{S^2_{n-2}(y)}$. Then
\begin{eqnarray*}
d^2+d + y-2 &=& f^2 - f  - 2ef+e^2+e  + y-2 \\
&=& \frac{1}{S^2_{n-2}(y)} + y -2  - 2ef+e^2+e. 
\end{eqnarray*}
Since $S^2_{n-2}(y) + S^2_{n-3}(y) -y S_{n-2}(y) S_{n-3}(y) =1$ we have
$$
\frac{1 + (y-2)S^2_{n-2}(y)} {S^2_{n-2}(y)} = \frac{[S_{n-2}(y) - S_{n-3}(y)] [(y-1)S_{n-2}(y) - S_{n-3}(y)]} {S^2_{n-2}(y)}.
$$
This implies that 
\begin{eqnarray*}
d^2+d + y-2 &=& e (y-2+e)  - 2ef+e^2+e \\
&=& e (y-1+2e-2f).
\end{eqnarray*}
Hence
\begin{eqnarray*}
Q' &=& e (y-1+2e-2f)+ (d^2+yd+1) ef S^2_{n-2}(y) \\
&=& S^2_{n-2}(y) [e (y-1+2e-2f) (f^2-f)+ (d^2+yd+1) ef] \\
&=& ef S^2_{n-2}(y) [(y-1+2e-2f) (f-1)+ (e-f)^2+y(e-f)+1] \\
&=& ef S^2_{n-2}(y) [ 1+ (e-1)^2+(e-1)y -(f^2-f)] \\
&=& ef S^2_{n-2}(y) \left[ 1 + \frac{S^2_{n-3}(y) }{S^2_{n-2}(y)} - \frac{S_{n-3}(y) }{S_{n-2}(y)} y - \frac{1}{S^2_{n-2}(y)} \right] \\
&=& ef [S^2_{n-2}(y) + S^2_{n-3}(y) -y S_{n-2}(y) S_{n-3}(y) -1] \\
&=& 0.
\end{eqnarray*}

We now prove that $R_n=0$. Note that 
$$
T= (z^2-1)^2  - S_{n-2}(y) (S_{n-2}(y) -S_{n-3}(y)) (z^2 - 1)  -  (S_{n-2}(y) -S_{n-3}(y))^2 = 0,
$$
since $z^2 - 1 = r (S_{n-2}(y) - S_{n-3}(y))$ and $r^2 - S_{n-2}(y) r -1 =0$. This implies that
$$
S = z^2q^2 - (yz^2+y-1)pq + (y^2+z^2-3)p^2 = (z^2-y^3+3y-2) T = 0.
$$

Case 1: $z \not= 0$. We have $z R_n = z R_n + S_{n-2}(y) Q=  -p x + q z$. 

If $p=0$, then since $z^2q^2 - (yz^2+y-1)pq + (y^2+z^2-3)p^2$ we get $q =0$.

If $p \not= 0$, then 
\begin{eqnarray*}
x &=& z \left( y+1-\frac{S_{n-3}(y)+r}{S_{n-2}(y)} \right) \\
&=& z \frac{y^2+y-1 - r S_{n-3}(y)  }{  y + r ( S_{n-2}(y) - y S_{n-3}(y))} \\
&=& z \frac{q}{p}.
\end{eqnarray*} 
In both cases we have $-p x + q z = 0$. Hence $R_n=0$.

Case 2: $z=0$. Then  $x =z(y+1-\frac{S_{n-3}(y)+r}{S_{n-2}(y)}) = 0$ and 
\begin{eqnarray*}
R_n &=& (y+2) S_{n-2}(y) - (y^2 + y -2)S_{n-3}(y) \\
&=& (y+2) [S_{n-2}(y) - (y-1)S_{n-3}(y)].
\end{eqnarray*}
 
Since $T=0$ and $z=0$, we have 
$$
1 + S_{n-2}(y) (S_{n-2}(y) -S_{n-3}(y))  -  (S_{n-2}(y) -S_{n-3}(y))^2 = 0.
$$ 
Combining with $1 = S^2_{n-2}(y) + S^2_{n-3}(y) -y S_{n-2}(y) S_{n-3}(y)$, we obtain 
$$
S_{n-2}(y) [ S_{n-2}(y)  - (y-1) S_{n-3}(y)  ]= 0 .
$$ 
Since $S_{n-2}(y) \not= 0$, we must have $S_{n-2}(y)  - (y-1) S_{n-3}(y) = 0$. Hence $R_n = 0$. 
\end{proof}

\section{Roots of the Alexander polynomial} \label{Alex}

In this section we find certain roots of the Alexander polynomial of $\CK_n$ on the unit circle, which will then be used in the next section to construct continuous paths of $\mathrm{SL}_2(\BR)$-representations of $\CK_n$.

Note that the $(-2,3,2n+1)$-pretzel knot $\CK_n$ is equal to the $(n-2)$-twisted $(3, 5)$-torus knot. By \cite{Tr-twistedtorus}, the Alexander polynomial of the $(n-2)$-twisted $(3, 3m+2)$-torus knot, with $n, m \ge 1$, is given by  $[t^{2n+6m-1} + 1- (t^{2n+3m-1}-t^{2})(t^{3m}-1)/(t^2+t+1)]/(t+1)$. By taking $m=1$ we obtain
$$
\Delta_{\CK_n}(t) = (t^{2n+5}-t^{2n+3}+t^{2n+2}+t^3-t^2+1)/(t+1).
$$

As in \cite[Section 3]{Tr-twistedtorus}, we will find certain solutions of $\Delta_{\CK_n}(t)=0$ on the unit circle. With $t = e^{i \theta}$ we have $\cos \frac{(2n+5)\theta}{2} - \cos \frac{(2n+1)\theta}{2} + \cos \frac{(2n-1)\theta}{2}  = 0$. This is equivalent to 
$$
2\cos(n+1)\theta \cos \frac{3\theta}{2}  - \cos \frac{(2n+1)\theta}{2}  = 0.
$$

For $n \ge 3$, we define an interval $[\alpha_n, \beta_n] \subset \BR$  by
$$
[\alpha_n, \beta_n] = \begin{cases}
			\left[ \frac{2\pi}{3}-\frac{\pi}{6n+15}, \frac{2\pi}{3} + \frac{\pi}{6n+3} \right] & \text{if } n \equiv 0 \pmod{3}\\
            \left[ \frac{2\pi}{3} + \frac{\pi}{6n-3}, \frac{2\pi}{3} +  \frac{\pi}{2n+1} \right] & \text{if } n \equiv 1 \pmod{3} \\
            \left[ \frac{2\pi}{3} - \frac{\pi}{2n+5}, \frac{2\pi}{3} -  \frac{\pi}{6n+3} \right] & \text{if } n \equiv 2 \pmod{3}
		 \end{cases}.
$$ 
We also define $\gamma_n \in (\alpha_n, \beta_n)$ by 
$$
\gamma_n = \begin{cases}
			\frac{2\pi}{3} & \text{if } n \equiv 0 \pmod{3}\\
           \frac{2\pi}{3} + \frac{4\pi}{12n+9} & \text{if } n \equiv 1 \pmod{3}\\
           \frac{2\pi}{3} - \frac{4\pi}{12n+9} & \text{if } n \equiv 2 \pmod{3}
		 \end{cases}.
$$ 

Let $F(\theta) : = 2\cos(n+1)\theta \cos \frac{3\theta}{2}  - \cos \frac{(2n+1)\theta}{2}$ for $\theta \in [\alpha_n, \beta_n]$.

\begin{proposition} \label{root}
The equation $F(\theta) = 0$ has a unique solution $\theta_n$ on $[\alpha_n, \beta_n]$. Moreover $\theta_n \in [\alpha_n, \gamma_n)$.
\end{proposition}

\begin{proof} 
Let $\ve = \begin{cases}
			\frac{2\pi}{3} - \theta & \text{if } n \not\equiv 1 \pmod{3}\\
            \theta - \frac{2\pi}{3}& \text{if } n \equiv 1 \pmod{3}
		 \end{cases}
$ and $\widetilde{F}(\ve) :=  F(\theta)$. 

Case 1: $n \equiv 0 \pmod{3}$. Then $\ve = \frac{2\pi}{3} - \theta \in \left[ -\frac{\pi}{6n+3}, \frac{\pi}{6n+15} \right]$ and 
\begin{eqnarray*}
\widetilde{F}(\ve)
&=& 2\cos \left( \frac{(2n+2)\pi}{3}-(n+1)\ve \right) \cos \left( \pi - \frac{3\ve}{2} \right) -\cos \left( \frac{(2n+1)\pi}{3}  - \frac{(2n+1)\ve}{2} \right) \\
&=& 2\cos \left( \frac{\pi}{3}+(n+1)\ve \right) \cos \frac{3\ve}{2}  -\cos \left( \frac{\pi}{3} - \frac{(2n+1)\ve}{2} \right).
\end{eqnarray*}
We claim that $\widetilde{F}(\ve) = 0$ has a unique solution in $\big( 0, \frac{\pi}{6n+15} \big]$ and no solutions in $\big[ - \frac{\pi}{6n+3}, 0 \big]$. 

If $\ve \in \left[ - \frac{\pi}{6n+3}, 0 \right]$, then $0 < \frac{\pi}{3}+(n+1)\ve \le \frac{\pi}{3}-\frac{(2n+1)\ve}{2} \le \frac{\pi}{2}$. Hence
$$
\widetilde{F}(\ve) \ge    \left( 2\cos \frac{3\ve}{2} - 1 \right)  \cos \left( \frac{\pi}{3}+(n+1)\ve\right) >0.
$$

If $\ve \in \big( 0,  \frac{\pi}{6n+15} \big]$, then $0 < \frac{\pi}{3}-\frac{(2n+1)\ve}{2} < \frac{\pi}{3}+(n+1)\ve < \frac{\pi}{2}$. It is easy to see that $\widetilde{F}(\ve)$ is a decreasing function. Note that $f(0) = \frac{1}{2}$. Since $\widetilde{F}(\ve)
= 2\sin \left( \frac{\pi}{6} - (n+1)\ve \right) \cos \frac{3\ve}{2}  -\sin\left( \frac{\pi}{6} + \frac{(2n+1)\ve}{2} \right)$ we have
\begin{eqnarray*}
\widetilde{F} \left( \frac{\pi}{6n+15} \right) &=& 2\sin \frac{\pi}{4n+10} \cos \frac{\pi}{4n+10}  -\sin\left( \frac{\pi}{6} + \frac{(2n+1)\pi}{12n+30} \right) \\
&=& \sin \frac{\pi}{2n+5}  -\sin\left( \frac{\pi}{6} + \frac{(2n+1)\pi}{12n+30} \right) \\
&=& - 2 \sin \frac{n\pi}{6n+15} \cos \frac{(n+3)\pi}{6n+15} \\
&<& 0. 
\end{eqnarray*}
Hence there exists a unique solution of $\widetilde{F}(\ve) = 0$ in $\left( 0,  \frac{\pi}{6n+15} \right]$. 

Case 2: $n \equiv 1 \pmod{3}$.  Then $\ve =  \theta - \frac{2\pi}{3}  \in \left[ \frac{\pi}{6n-3}, \frac{\pi}{2n+1} \right]$ and 
\begin{eqnarray*}
\widetilde{F}(\ve) 
&=& 2\cos \left( \frac{(2n+2)\pi}{3}+(n+1)\ve \right) \cos \left( \pi +\frac{3\ve}{2} \right) -\cos \left( \frac{(2n+1)\pi}{3}  + \frac{(2n+1)\ve}{2} \right) \\
&=& - 2\cos \left( \frac{2\pi}{3} - (n+1)\ve \right) \cos \frac{3\ve}{2}  +\cos \frac{(2n+1)\ve}{2}.
\end{eqnarray*}
We claim that $\widetilde{F}(\ve) =0$ has a unique solution in $\big[ \frac{\pi}{6n-3}, \frac{4\pi}{12n+9} \big)$ and no solutions in $\big[  \frac{4\pi}{12n+9}, \frac{\pi}{2n+1} \big]$. Note that $\frac{2\pi}{3} - (n+1)\ve = \frac{(2n+1)\ve}{2}$ if $\ve = \frac{4\pi}{12n+9}$. 

If $\ve \in \big[  \frac{4\pi}{12n+9}, \frac{\pi}{2n+1} \big]$, then $0 < \frac{2\pi}{3} - (n+1)\ve \le \frac{(2n+1)\ve}{2} \le \frac{\pi}{2}$. Hence
$$
\widetilde{F}(\ve) \le  \left( 1-2 \cos \frac{3\ve}{2} \right) \cos \left( \frac{2\pi}{3} - (n+1)\ve \right) < 0.
$$

If $\ve \in \big[ \frac{\pi}{6n-3},  \frac{4\pi}{12n+9} \big)$, then $0  < \frac{3\ve}{2} < \frac{(2n+1)\ve}{2} < \frac{2\pi}{3} - (n+1)\ve < \frac{\pi}{2}$. We have 
\begin{eqnarray*}
\frac{d\widetilde{F}}{d\ve} &=& -2 (n+1) \sin \left( \frac{2\pi}{3} - (n+1)\ve \right) \cos \frac{3\ve}{2}
+ 3\cos \left( \frac{2\pi}{3} - (n+1)\ve \right) \sin \frac{3\ve}{2}  \\
&& -  \, \frac{2n+1}{2} \sin  \frac{(2n+1)\ve}{2}\\
 &<&-3\sin \left( \frac{2\pi}{3} - (n+1)\ve \right) \cos \frac{3\ve}{2}
+ 3\cos \left( \frac{2\pi}{3} - (n+1)\ve \right) \sin \frac{3\ve}{2}  \\
&& -  \, \frac{2n+1}{2} \sin  \frac{(2n+1)\ve}{2}\\
 &=& - \left[ 3 \sin \left( \frac{2\pi}{3} - (n+1)\ve - \frac{3\ve}{2}   \right) +  \frac{2n+1}{2} \sin  \frac{(2n+1)\ve}{2} \right] \\
 &<& 0.
\end{eqnarray*}
Then $\widetilde{F}(\ve)$ is a decreasing function. Note that $\widetilde{F} \left( \frac{4\pi}{12n+9} \right) < 0$.

Since $\widetilde{F} (\ve) = - 2\sin \left( (n+1)\ve - \frac{\pi}{6} \right) \cos \frac{3\ve}{2}  +\sin (\frac{\pi}{2} - \frac{(2n+1)\ve}{2}) $ we have  
\begin{eqnarray*}
\widetilde{F}  \left( \frac{\pi}{6n-3} \right) 
&=& - 2\sin \frac{\pi}{4n-2}  \cos \frac{\pi}{4n-2}  +\sin \left( \frac{\pi}{2} - \frac{(2n+1)\pi}{12n-6} \right) \\
&=& -\sin \frac{\pi}{2n-1} +\sin \left( \frac{\pi}{2} - \frac{(2n+1)\pi}{12n-6} \right) \\
&=& 2 \sin \frac{(2n-5)\pi}{12n-6} \cos \frac{(2n+1)\pi}{12n-6} \\
&>& 0.
\end{eqnarray*}
Hence there exists a unique solution of $\widetilde{F}(\ve) = 0$ in $\big[ \frac{\pi}{6n-3},  \frac{4\pi}{12n+9} \big)$. 

Case 3: $n \equiv 2 \pmod{3}$.  Then $\ve = \frac{2\pi}{3} - \theta \in \left[ \frac{\pi}{6n+3}, \frac{\pi}{2n+5} \right]$ and 
\begin{eqnarray*}
\widetilde{F} (\ve)
&=& 2\cos \left( \frac{(2n+2)\pi}{3}-(n+1)\ve \right) \cos \left( \pi - \frac{3\ve}{2} \right) -\cos \left( \frac{(2n+1)\pi}{3}  - \frac{(2n+1)\ve}{2} \right) \\
&=& - 2\cos \left(  (n+1)\ve \right) \cos \frac{3\ve}{2}  + \cos \left( \frac{2\pi}{3} - \frac{(2n+1)\ve}{2} \right).
\end{eqnarray*}
We claim that $\widetilde{F}(\ve) =0$ has a unique solution in $\big( \frac{4\pi}{12n+9}, \frac{\pi}{2n+5} \big]$ and no roots in $\big[ \frac{\pi}{6n+3}, \frac{4\pi}{12n+9} \big]$.

If $\ve \in \big[\frac{\pi}{6n+3}, \frac{4\pi}{12n+9} \big]$, then $0  < (n+1)\ve \le \frac{2\pi}{3}-\frac{(2n+1)\ve}{2} \le \frac{\pi}{2}$. 
Hence
$$
\widetilde{F} (\ve) \le   \left( 1- 2\cos \frac{3\ve}{2}  \right)  \cos \left( (n+1)\ve \right) < 0.
$$

If $\ve \in \big( \frac{4\pi}{12n+9},  \frac{\pi}{2n+5} \big]$, then $0 < \frac{2\pi}{3}-\frac{(2n+1)\ve}{2} < (n+1)\ve < \frac{\pi}{2}$. 
It is easy to see that $\widetilde{F} (\ve)$ is an increasing function. Note that $\widetilde{F} \left( \frac{4\pi}{12n+9} \right) < 0$.

Since $\widetilde{F} (\ve) = - 2\sin  \left( \frac{\pi}{2} -  (n+1)\ve \right) \cos \frac{3\ve}{2}  + \sin \left( \frac{(2n+1)\ve}{2} - \frac{\pi}{6} \right)$ we have 
\begin{eqnarray*}
\widetilde{F} \left( \frac{\pi}{2n+5}\right) &=& - 2\sin \frac{3\pi}{4n+10} \cos \frac{3\pi}{4n+10}  + \sin \left( \frac{(2n+1)\pi}{4n+10} - \frac{\pi}{6} \right) \\
&=& - \sin \frac{3\pi}{2n+5}  + \sin \left( \frac{(2n+1)\pi}{4n+10} - \frac{\pi}{6} \right) \\
&=& 2 \sin \frac{(n-5)\pi}{6n+15} \cos \frac{(n+4)\pi}{6n+15}  \\
&\ge& 0.
\end{eqnarray*}
Hence there exists a unique solution of $\widetilde{F}(\ve) = 0$ in $\big( \frac{4\pi}{12n+9},  \frac{\pi}{2n+5} \big]$. 
\end{proof}

\begin{lemma} \label{sign}
The following hold true.

$(1)$ If $n \equiv 0 \pmod{3}$, then $S_{n-2}(2\cos\theta) < 0$ and $S_{n-2}(2\cos\theta) - S_{n-3}(2\cos\theta) < 0$ on $[\theta_n, \beta_n]$.

$(2)$ If $n \not\equiv 0 \pmod{3}$, then $S_{n-2}(2\cos\theta) > 0$ and $S_{n-2}(2\cos\theta) - S_{n-3}(2\cos\theta) > 0$ on $[\theta_n, \beta_n]$. 
\end{lemma}

\begin{proof}
By Lemma \ref{sign} we have 
\begin{eqnarray*}
S_{n-2}(y) &=& \prod_{k=1}^{n-2} \left( y-2\cos \frac{k\pi}{n-1} \right), \\ S_{n-2}(y) - S_{n-3}(y) &=& \prod_{l=1}^{n-2} \left( y-2\cos \frac{(2l-1)\pi}{2n-3} \right). 
\end{eqnarray*}
We will verify the sign for $S_{n-2}(2\cos\theta)$. The sign for $S_{n-2}(2\cos\theta) - S_{n-3}(2\cos\theta)$ is verified in a similar way.

If $n \equiv 0 \pmod{3}$, then $ [\theta_n, \beta_n] \subset \left[ \frac{2\pi}{3}-\frac{\pi}{6n+15}, \frac{2\pi}{3} + \frac{\pi}{6n+3} \right] \subset (\frac{2\pi}{3} -\frac{\pi}{3n-3}, \frac{2\pi}{3} + \frac{2\pi}{3n-3}) = (\frac{(\frac{2n}{3}-1)\pi}{n-1}, \frac{\frac{2n}{3} \pi}{n-1})$. Lemma \ref{sign} then implies that $S_{n-2}(2\cos\theta) < 0$. 

If $n \equiv 1 \pmod{3}$, then $ [\theta_n, \beta_n] \subset \left[ \frac{2\pi}{3} + \frac{\pi}{6n-3}, \frac{2\pi}{3} +  \frac{\pi}{2n+1} \right] \subset (\frac{2\pi}{3}, \frac{2\pi}{3} + \frac{\pi}{n-1}) = (\frac{\frac{2n-2}{3}\pi}{n-1}, \frac{\frac{2n+1}{3} \pi}{n-1})$. Hence $S_{n-2}(2\cos\theta) > 0$. 

If $n \equiv 2 \pmod{3}$, then $ [\theta_n, \beta_n] \subset \left[ \frac{2\pi}{3} - \frac{\pi}{2n+5}, \frac{2\pi}{3} -  \frac{\pi}{6n+3} \right]  \subset (\frac{2\pi}{3} - \frac{2\pi}{3n-3}, \frac{2\pi}{3} + \frac{\pi}{3n-3}) = (\frac{\frac{2n-4}{3}\pi}{n-1}, \frac{\frac{2n-1}{3} \pi}{n-1})$. Hence $S_{n-2}(2\cos\theta) > 0$. 
\end{proof}

\begin{lemma} \label{Sn}
The following hold true.

$(1)$ $S_{n}(2\cos\theta) - S_{n-1}(2\cos\theta) \not= 0$ on $[\theta_n, \beta_n)$. Moreover $S_{n}(2\cos\theta) - S_{n-1}(2\cos\theta) = 0$ at $\theta = \beta_n$. 

$(2)$ $S_{n}(2\cos\theta) + S_{n-1}(2\cos\theta) \not= 0$ on $[\theta_n, \beta_n]$. 
\end{lemma}

\begin{proof}
The proof is similar to that of Lemma \ref{sign} and hence is omitted.
\end{proof}

We end this section by  the following lemma which will be used in Section \ref{long}.

\begin{lemma} \label{lem}
Suppose $n \ge 3$ and $n \not= 4$. Let 
$$
G(\theta) := S^2_{n-1}(y) + ( S_{n-1}(y) - S_{n-2}(y) ) S^3_{n-2}(y)
$$ 
for $\theta \in [\theta_n, \beta_n]$, where $y=2\cos\theta$. Then $G < 0$ on $[\theta_n, \beta_n]$. 
\end{lemma}

\begin{proof}
The proof is divided into two steps.

\underline{Step 1}: We first claim that $G(\theta_n) < 0$. Indeed, at $\theta = \theta_n$ we have  $t^{2n+5}-t^{2n+3}+t^{2n+2}+t^3-t^2+1 =0$ where $t := e^{i \theta}$. Since  $y = 2\cos\theta = t+t^{-1}$, by a direct calculation using $S_{k}(y) = (t^{k+1} - t^{-k-1}) / (t - t^{-1})$ we have 
$$
(y-1)(y^2-3) S_{n-1}(y) -  (y^2-y-1) S_{n-2}(y) = \frac{t^{2n+5}-t^{2n+3}+t^{2n+2}+t^3-t^2+1}{t^{n+2}(t+1)} = 0. 
$$
This implies that $S_{n-1}(y) = \frac{y^2-y-1}{(y-1)(y^2-3)} S_{n-2}(y)$. Combining with $S^2_{n-1}(y) + S^2_{n-2}(y) - y S_{n-1}(y) S_{n-2}(y) =1$ we obtain
$S^2_{n-2}(y) = \frac{(y-1)^2(y^2-3)^2}{(2-y)(y^3 -y^2 -4y +5)}$. Hence
$$
G(\theta_n) = \frac{(y^6-2 y^5-5 y^4+11 y^3+6 y^2-19 y+7)(y^5-y^4-5 y^3+6 y^2+5 y-7)}{(2-y)(y^3-y^2-4 y+5)^2}. 
$$

The real roots of $H_1(y) := y^6-2 y^5-5 y^4+11 y^3+6 y^2-19 y+7$ are approximately $0.50204$ and $1.47949$, so $H_1(y) > 0$ (since $y = 2 \cos\theta_n< 0$).  

The real roots of the polynomial $H_2(y) : = y^5 - y^4 - 5 y^3 + 6 y^2 + 5 y - 7$ are approximately $-1.96757$, $-1.22062$ and  $1.66787$. 

If $n \not\equiv 1 \pmod{3}$, then by Proposition \ref{root} we have $\theta_n < \gamma_n \le \frac{2\pi}{3}$. So $ y > 2 \cos \frac{2\pi}{3}= -1$. 

If $n \equiv 1 \pmod{3}$ then  $\theta_n \in (\alpha_n, \gamma_n) = (  \frac{2\pi}{3} + \frac{\pi}{6n-3}, \frac{2\pi}{3} +  \frac{4\pi}{12n+9} )$. If $n \ge 10$, then $y > 2 \cos (\frac{2\pi}{3} + \frac{4\pi}{129}) \approx -1.16372$. If $n=7$, then by a direct calculation we have $\theta_n \approx 2.20391$ and so $y \approx -1.18332$. 

Since $y \in (-1.22062, 1.66787)$ in all cases, we have $H_2(y) < 0$. 
Hence $G(\theta_n) <0$. 

\underline{Step 2}: We now prove that $G < 0$ on $[\theta_n, \beta_n]$. Since $S_k(y)  = \frac{\sin(k+1)\theta}{\sin  \theta}$ we have 
$$
G(\theta)=  \frac{\sin^2(n\theta) \sin^2\theta + 2 \sin \frac{\theta}{2} \cos ((n - \frac{1}{2})\theta)\sin^3((n-1)\theta)} {\sin^4 \theta}.
$$

Let $\ve = \begin{cases}
			\frac{2\pi}{3} - \theta & \text{if } n \not\equiv 1 \pmod{3}\\
            \theta - \frac{2\pi}{3}& \text{if } n \equiv 1 \pmod{3}
		 \end{cases}
$ and $\widetilde{G}(\ve) :=  G(\theta) \sin^4 \theta$. 

It is equivalent to show that $\widetilde{G}(\ve) < 0$ for $\ve \in [\frac{2\pi}{3} - \beta_n, \frac{2\pi}{3} - \theta_n]$ if $n \not\equiv 1 \pmod{3}$ and $\ve \in [\theta_n - \frac{2\pi}{3}, \beta_n - \frac{2\pi}{3}]$ if $n \equiv 1 \pmod{3}$. There are three cases to consider. 

\no{
\begin{eqnarray*}
&&\begin{cases}
			\sin^2 (n \varepsilon) \sin^2 (\frac{\pi}{3}-\varepsilon) - 2 \sin (\frac{\pi}{3} + \frac{\varepsilon}{2}) \cos (\frac{\pi}{3} - (n - \frac{1}{2})\varepsilon) \sin^3(\frac{\pi}{3}+(n-1)\varepsilon) & \text{if } n \equiv 0 \pmod{3}\\
           \sin^2 (\frac{\pi}{3} - n \varepsilon) \sin^2 (\frac{\pi}{3}-\varepsilon)  - 2 \sin (\frac{\pi}{3} + \frac{\varepsilon}{2})  \sin ((n-\frac{1}{2})\varepsilon - \frac{\pi}{6})\sin^3 (n-1)\varepsilon& \text{if } n \equiv 1 \pmod{3}\\
           \sin^2 (\frac{\pi}{3} + n \varepsilon) \sin^2 (\frac{\pi}{3}-\varepsilon)  - 2 \sin (\frac{\pi}{3}+\frac{\varepsilon}{2}) \cos ((n-\frac{1}{2})\varepsilon) \sin^3 (\frac{2\pi}{3} + (n-1)\varepsilon)& \text{if } n \equiv 2 \pmod{3}
		 \end{cases}.
\end{eqnarray*}
Then $(\sin^4 \theta)  G(\theta) = (\sin^4 \theta) \widetilde{G}(\ve)$. Since $G(\theta_n) < 0$ we have $\widetilde{G}(\theta_n - \frac{2\pi}{3}) < 0$. 

It suffices to show that $\widetilde{G}(\ve) < 0$ for $\ve \in (\theta_n - \frac{2\pi}{3}, \beta_n - \frac{2\pi}{3})$. There are $3$ cases to consider. 
}

Case 1: $n \equiv 0 \pmod{3}$. Then $\ve = \frac{2\pi}{3} - \theta \in [\frac{2\pi}{3} - \beta_n, \frac{2\pi}{3} - \theta_n] \subset \left[-\frac{\pi}{6n+3}, \frac{\pi}{6n+15} \right]$ and 
\begin{eqnarray*}
 \widetilde{G}(\varepsilon) &=& 
 \sin^2(\frac{2n\pi}{3} - n\ve ) \sin^2(\frac{2\pi}{3} - \ve) \\
 && + \, 2 \sin (\frac{\pi}{3} -  \frac{\ve}{2}) \cos ((n - \frac{1}{2}) \frac{2\pi}{3} - (n - \frac{1}{2})\ve) \sin^3((n-1) \frac{2\pi}{3} - (n-1)\ve) \\
 &=&  \sin^2 (n \varepsilon) \sin^2 (\frac{\pi}{3}+\varepsilon) - 2 \sin (\frac{\pi}{3} - \frac{\varepsilon}{2}) \cos (\frac{\pi}{3} + (n - \frac{1}{2})\varepsilon) \sin^3(\frac{\pi}{3}-(n-1)\varepsilon).
\end{eqnarray*}

If $\varepsilon \in [0, \frac{2\pi}{3} - \theta_n] \subset \left[0, \frac{\pi}{6n+15} \right]$ then $\widetilde{G}(\varepsilon)$ is an increasing function, so 
$$
\widetilde{G}(\varepsilon) \le \widetilde{G}(\frac{2\pi}{3} - \theta_n ) = G(\theta_n) \sin^4 \theta_n < 0.
$$ 

If $\varepsilon \in [\frac{2\pi}{3} - \beta_n, 0] = [-\frac{\pi}{6n+3},0]$, then $\sin^2 (n \varepsilon)  = \sin^2 (n |\varepsilon|) \le \sin^2 (\frac{n\pi}{6n+3})$. Hence 
\begin{eqnarray*}
\widetilde{G}(\varepsilon) 
&<& \sin^2 (\frac{n\pi}{6n+3}) \sin^2 (\frac{\pi}{3} - |\varepsilon|) - 2 \sin (\frac{\pi}{3} + \frac{|\varepsilon|}{2}) \cos (\frac{\pi}{3} -(n - \frac{1}{2})|\varepsilon|) \sin^3(\frac{\pi}{3} + (n-1)|\varepsilon|) \\
&<& \sin^2 (\frac{n\pi}{6n+3}) \sin^2 (\frac{\pi}{3} ) - 2 \sin (\frac{\pi}{3} ) \cos (\frac{\pi}{3})\sin^3(\frac{\pi}{3} ) \\
&=& \sin^2 (\frac{n\pi}{6n+3})\sin^2 (\frac{\pi}{3} ) -  \sin^4(\frac{\pi}{3} ) \\ 
&<& 0.
\end{eqnarray*}

\no{
$$
\widetilde{G}(\varepsilon) <  \sin^2 (\frac{\pi}{6}) \sin^2 (\frac{\pi}{3} - \varepsilon) - 2 \sin (\frac{\pi}{3} + \frac{\varepsilon}{2}) \cos (\frac{\pi}{3} - (n - \frac{1}{2})\varepsilon) \sin^3(\frac{\pi}{3}+(n-1)\varepsilon).
$$ 
The latter is a decrasing increasing function in $\varepsilon$. Hence
\begin{eqnarray*}
\widetilde{G}(\varepsilon) &<& \sin^2 (\frac{\pi}{6}) \sin^2 (\frac{\pi}{3} ) - 2 \sin (\frac{\pi}{6} ) \cos (\frac{\pi}{6})\sin^3(\frac{\pi}{3} ) \\
&=& \sin^2 (\frac{\pi}{6}) \sin^2 (\frac{\pi}{3} ) -  \sin^4(\frac{\pi}{3} ) < 0.
\end{eqnarray*}

As $|\varepsilon|$ increases, $\tilde{g}(|\varepsilon|)$ decreases. Hence 
\begin{eqnarray*}
g(|\ve|) &\le& \tilde{g}(0) \\
&=&  \sin^2 (\frac{\pi}{6}) \sin^2 (\frac{\pi}{3} ) - 2 \sin (\frac{\pi}{6} ) \cos (\frac{\pi}{6})\sin^3(\frac{\pi}{3} ) \\
&=& \sin^2 (\frac{\pi}{6}) \sin^2 (\frac{\pi}{3} ) -  \sin^4(\frac{\pi}{3} ) < 0.
\end{eqnarray*}
}

Case 2: $n \equiv 1 \pmod{3}$ and $n \not= 4$. Then $\ve = \theta -\frac{2\pi}{3} \in [\theta_n -  \frac{2\pi}{3}, \beta_n -  \frac{2\pi}{3}] \subset (\frac{\pi}{6n-3}, \frac{\pi}{2n+1}]$ and
\begin{eqnarray*}
 \widetilde{G}(\varepsilon) &=& 
 \sin^2(\frac{2n\pi}{3} + n\ve ) \sin^2(\frac{2\pi}{3} + \ve) \\
 && + \, 2 \sin (\frac{\pi}{3} +  \frac{\ve}{2}) \cos ((n - \frac{1}{2}) \frac{2\pi}{3} + (n - \frac{1}{2})\ve) \sin^3((n-1) \frac{2\pi}{3} +(n-1)\ve) \\
 &=&  \sin^2 (\frac{\pi}{3} - n \varepsilon) \sin^2 (\frac{\pi}{3}-\varepsilon)  + 2 \sin (\frac{\pi}{3} + \frac{\varepsilon}{2}) \cos ( \frac{\pi}{3} + (n-\frac{1}{2})\varepsilon)\sin^3 (n-1)\varepsilon \\
 &=&  \sin^2 (\frac{\pi}{3} - n \varepsilon) \sin^2 (\frac{\pi}{3}-\varepsilon)  - 2 \cos (\frac{\pi}{6} -\frac{\varepsilon}{2}) \sin  ((n-\frac{1}{2})\varepsilon - \frac{\pi}{6} )\sin^3 (n-1)\varepsilon.
\end{eqnarray*}

If $\varepsilon \in [\theta_n -  \frac{2\pi}{3}, \frac{\pi}{3n}] \subset (\frac{\pi}{6n-3}, \frac{\pi}{3n}]$ then $\widetilde{G}(\varepsilon)$ is a decreasing function, so 
$$
\widetilde{G}(\varepsilon) \le \widetilde{G}(\theta_n -\frac{2\pi}{3} ) = G(\theta_n) \sin^4 \theta_n < 0.
$$ 

If  $\varepsilon \in [\frac{\pi}{3n}, \beta_n -  \frac{2\pi}{3}] \subset [\frac{\pi}{3n},  \frac{\pi}{2n+1}]$ then $\sin^2 (n \varepsilon - \frac{\pi}{3}) \le  \sin^2 (\frac{n\pi}{2n+1}- \frac{\pi}{3}) = \sin^2 ( \frac{\pi}{6} - \frac{\pi}{4n+2})$. Hence
\begin{eqnarray*}
\widetilde{G}(\varepsilon) &\le& \sin^2 ( \frac{\pi}{6} - \frac{\pi}{4n+2})\sin^2 (\frac{\pi}{3}-\varepsilon)  - 2 \cos (\frac{\pi}{6} -\frac{\varepsilon}{2}) \sin  ((n-\frac{1}{2})\varepsilon - \frac{\pi}{6} )\sin^3 (n-1)\varepsilon \\
&\le& \sin^2 ( \frac{\pi}{6} - \frac{\pi}{4n+2}) \sin^2 (\frac{\pi}{3}-\frac{\pi}{3n})  - 2 \cos (\frac{\pi}{6}-\frac{\pi}{6n}) \sin (\frac{(n-\frac{1}{2})\pi}{3n} - \frac{\pi}{6})\sin^3 \frac{(n-1)\pi}{3n} \\
&=& \sin^2 ( \frac{\pi}{6} - \frac{\pi}{4n+2})\sin^2 (\frac{\pi}{3}-\frac{\pi}{3n}) - \sin^4 (\frac{\pi}{3}-\frac{\pi}{3n}) \\
&<& 0.
\end{eqnarray*}

Case 3: $n \equiv 2 \pmod{3}$. Then $\ve = \frac{2\pi}{3} - \theta \in [\frac{2\pi}{3} - \beta_n, \frac{2\pi}{3} - \theta_n] \subset \left[\frac{\pi}{6n+3}, \frac{\pi}{2n+5} \right]$ and 
\begin{eqnarray*}
 \widetilde{G}(\varepsilon) &=& 
 \sin^2(\frac{2n\pi}{3} - n\ve ) \sin^2(\frac{2\pi}{3} - \ve) \\
 && + \, 2 \sin (\frac{\pi}{3} -  \frac{\ve}{2}) \cos ((n - \frac{1}{2}) \frac{2\pi}{3} - (n - \frac{1}{2})\ve) \sin^3((n-1) \frac{2\pi}{3} - (n-1)\ve) \\
 &=&  \sin^2 (n \varepsilon - \frac{\pi}{3} ) \sin^2 (\frac{\pi}{3}+\varepsilon)  - 2 \sin (\frac{\pi}{3}-\frac{\varepsilon}{2}) \cos ((n-1/2)\varepsilon) \sin^3 (\frac{2\pi}{3} - (n-1)\varepsilon).
\end{eqnarray*}

If $\varepsilon \in [\frac{\pi}{3n}, \frac{2\pi}{3} - \theta_n] \subset [\frac{\pi}{3n}, \frac{\pi}{2n+5}] $ then $\widetilde{G}(\varepsilon)$ is an increasing function, so 
$$
\widetilde{G}(\varepsilon) \le \widetilde{G}(\frac{2\pi}{3} - \theta_n ) = G(\theta_n) \sin^4 \theta_n < 0.
$$ 

If $\varepsilon \in [\frac{2\pi}{3} - \beta_n, \frac{\pi}{3n}] = [\frac{\pi}{6n+3}, \frac{\pi}{3n}] $ then $\frac{2\pi}{3} - (n-1)\varepsilon \in [\frac{\pi}{3} + \frac{\pi}{3n}, \frac{\pi}{2} + \frac{\pi}{4n+2}]$, so 
\begin{eqnarray*}
\sin(\frac{2\pi}{3} - (n-1)\varepsilon) &\ge& \min \left( \sin (\frac{\pi}{3} + \frac{\pi}{3n}), \sin(\frac{\pi}{2} + \frac{\pi}{4n+2}) \right) \\
&=& \min \left( \sin (\frac{\pi}{3} + \frac{\pi}{3n}), \sin(\frac{\pi}{2} - \frac{\pi}{4n+2}) \right) \\
&=& \sin (\frac{\pi}{3} + \frac{\pi}{3n}).
\end{eqnarray*}
Note that $\sin^2 (\frac{\pi}{3} - n \varepsilon) \le \sin^2 (\frac{\pi}{3} - \frac{n\pi}{6n+3})$. Hence 
\begin{eqnarray*}
\widetilde{G}(\varepsilon)  &\le & \sin^2 (\frac{\pi}{3} - \frac{n\pi}{6n+3}) \sin^2 (\frac{\pi}{3}+\varepsilon)  - 2 \sin (\frac{\pi}{3}-\frac{\varepsilon}{2}) \cos ((n-1/2)\varepsilon) \sin^3 (\frac{\pi}{3} + \frac{\pi}{3n}) \\
&\le&  \sin^2 (\frac{\pi}{3} - \frac{n\pi}{6n+3})\sin^2 (\frac{\pi}{3}+\frac{\pi}{3n}) - 2 \sin (\frac{\pi}{3}-\frac{\pi}{6n}) \cos (\frac{\pi}{3}-\frac{\pi}{6n}) \sin^3 (\frac{\pi}{3} + \frac{\pi}{3n}) \\
&=&  \sin^2 (\frac{\pi}{3} - \frac{n\pi}{6n+3})\sin^2 (\frac{\pi}{3}+\frac{\pi}{3n}) -  \sin (\frac{2\pi}{3}-\frac{\pi}{3n})  \sin^3 (\frac{\pi}{3} + \frac{\pi}{3n}) \\
&=&   \sin^2 (\frac{\pi}{3} - \frac{n\pi}{6n+3}) \sin^2 (\frac{\pi}{3}+\frac{\pi}{3n}) -\sin^4 (\frac{\pi}{3} + \frac{\pi}{3n}) \\
&<& 0.
\end{eqnarray*}
This completes the proof of Lemma \ref{lem}.
\end{proof}

\begin{remark}
Lemma \ref{lem} does not hold true for $n=4$. Indeed, if $n=4$ then $\theta_n \approx 2.2728$ and $\beta_n \approx 2.44346$. At $\theta = \theta_n$, we have $y \approx -1.2915$ and so $G(\theta_n) = y^9-y^8-5 y^7+5 y^6+9 y^5-10 y^4-7 y^3+8 y^2+2 y-1 \approx 0.112614$. 
However, at $\theta = 2.3 \in (\theta_n, \beta_n)$ we have $y \approx -1.33255$ and $G(\theta) \approx -0.133199$. Hence the sign of $G$ changes on $[\theta_n, \beta_n]$.
\end{remark}

\section{Continuous paths of $\mathrm{SL}_2(\BR)$-representations} \label{path}

For $\theta \in [\theta_n, \beta_n]$ we let $y := 2 \cos \theta$ and 
$$
r: = \begin{cases}
	    (S_{n-2}(y) - \sqrt{S^2_{n-2}(y) + 4})/2 & \text{if } n \equiv 0 \pmod{3}, \\
            (S_{n-2}(y) +  \sqrt{S^2_{n-2}(y) + 4})/2 & \text{if } n \not\equiv 0 \pmod{3}.
	\end{cases}
$$
Then $r - r^{-1} = S_{n-2}(y) \not= 0$. Moreover, by Proposition \ref{sign} we have  $r < -1$ if $n \equiv 0 \pmod{3}$, and $r > 1$ if $n \not\equiv 0 \pmod{3}$. Hence $|r| > 1$. 

Let $z_{\pm} =  \pm \sqrt{1+ r (S_{n-2}(y) - S_{n-3}(y))}$ and $x_{\pm} = z_{\pm}  (y+1 - \frac{S_{n-3}(y) + r}{S_{n-2}(y)} )$. By Proposition \ref{sol} we have  $Q=R_n=0$. Hence there exists a representation $\rho_{\pm}: G(\CK_n) \to \mathrm{SL}_2(\BC)$ such that $\tr \rho_{\pm}(a) = x_{\pm}$, $\tr \rho_{\pm}(w) = y$ and $\tr \rho_{\pm}(aw) = z_{\pm}$. Note that $|z|>1$.

\begin{proposition} \label{x}
$x_{\pm}= \mp 2$ at $\theta = \beta_n$.
\end{proposition}

\begin{proof}
At $\theta =\beta_n$, by Lemma \ref{Sn} we have $S_n(y) - S_{n-1}(y)=0$. Since $S_{n}(y) - S_{n-1}(y)= (y-1) S_{n-1}(y) - S_{n-2}(y)$ we get $S_{n-1}(y) = \frac{1}{y-1} S_{n-2}(y)$. Then $S_{n-3}(y) =  y S_{n-2}(y) - S_{n-1}(y) = \frac{y^2-y-1}{y-1}S_{n-2}(y)$. Since $S^2_{n-2}(y) + S^2_{n-3}(y) - y S_{n-2}(y) S_{n-3}(y) =1$ we obtain
$$
\frac{2-y}{(1-y)^2}S^2_{n-2}(y) =1. 
$$

If $n \equiv 0 \pmod{3}$, then $S_{n-2}(y) < 0$. Hence $S_{n-2}(y) = \frac{y-1}{\sqrt{2-y}}$ and $S_{n-3}(y) =  \frac{y^2-y-1}{\sqrt{2-y}}$, so $r = (S_{n-2}(y) - \sqrt{S^2_{n-2}(y) + 4})/2 = - \sqrt{2-y}$.

If $n \not\equiv 0 \pmod{3}$, then $S_{n-2}(y) > 0$. Hence $S_{n-2}(y) = \frac{1-y}{\sqrt{2-y}}$ and $S_{n-3}(y) = \frac{1+y-y^2}{\sqrt{2-y}}$, so $r = (S_{n-2}(y) +\sqrt{S^2_{n-2}(y) + 4})/2 = \sqrt{2-y}$.

In both cases we have $z_{\pm} =  \pm \sqrt{1+ r (S_{n-2}(y) - S_{n-3}(y))} = \pm (1-y)$ and $x_{\pm} = z_{\pm} (y+1 - \frac{S_{n-3}(y) + u}{S_{n-2}(y)} ) = \pm (1-y) \frac{2}{y-1} = \mp 2$.
\end{proof}

Let $s :=  x^2+y^2+z^2-xyz-4$. Note that $Q=x-xy+ (s +1)z=0$. 

\begin{proposition}  \label{s}
$s > 0$ on $(\theta_n, \beta_n]$. Moreover $s=0$ at $\theta = \theta_n$.
\end{proposition}

\begin{proof}
At $\theta = \beta_n$ we have $x_{\pm} = \mp 2$ (by Proposition \ref{x}), so $s = (y \pm z)^2 \ge  0$.

Assume that $s=0$ for some $\theta \in  (\theta_n, \beta_n]$. 
Then $Q=x-xy+z=0$, so $z=(y-1)x=(y-1)z (y+1 - \frac{S_{n-3}(y) + r}{S_{n-2}(y)} )$. Hence $(y-1)(y+1 - \frac{S_{n-3}(y) + r}{S_{n-2}(y)} ) = 1$. This implies that $r = S_{n-2}(y) \frac{y^2-2}{y-1} - S_{n-3}(y)$. Since $r - r^{-1}= S_{n-2}(y)$ we get
\begin{equation} \label{eqn}
\left( S_{n-2}(y) \frac{y^2-2}{y-1} - S_{n-3}(y) \right)^2 - S_{n-2}(y) \left( S_{n-2}(y) \frac{y^2-2}{y-1} - S_{n-3}(y) \right) - 1 =0.
\end{equation}

Let $t=e^{i\theta}$. Since $y = 2 \cos\theta = t+t^{-1}$ we have $S_{k}(y) = (t^{k+1} - t^{-k-1}) / (t - t^{-1})$. By a direct calculation, equation \eqref{eqn} becomes 
$$
(t^{2n-2} - 1)(t^{2n+5}-t^{2n+3}+t^{2n+2}+t^3-t^2+1) =0. 
$$

If $t^{2n-2} - 1=0$, then $S_{n-2}(y)=0$ which contradicts $S_{n-2}(y) \not= 0$ on $[\theta_n, \beta_n]$. 

 If $t^{2n+5}-t^{2n+3}+t^{2n+2}+t^3-t^2+1 =0$, then $\Delta_{\CK_n}(e^{i\theta}) = 0$. This contradicts the fact that the equation $\Delta_{\CK_n}(e^{i\theta}) = 0$ has no solutions in $(\theta_n, \beta_n]$. Hence $s > 0$ in $(\theta_n, \beta_n]$. 

We now prove that $s(\theta_n) =0$. At $\theta = \theta_n$ we have $\Delta_{\CK_n}(e^{i\theta_n}) = 0$, so  equation \eqref{eqn} holds true. This implies that $r = S_{n-2}(y) \frac{y^2-2}{y-1} - S_{n-3}(y)$ or $-r^{-1} = S_{n-2}(y) \frac{y^2-2}{y-1} - S_{n-3}(y)$. 

Assume that $-r^{-1} = S_{n-2}(y) \frac{y^2-2}{y-1} - S_{n-3}(y)$. Then 
\begin{eqnarray*}
S_{n-3}(y) &=& S_{n-2}(y) \frac{y^2-2}{y-1} + r^{-1} \\
&=& (r - r^{-1}) \frac{y^2-2}{y-1} + r^{-1} \\
&=&  \frac{y^2-2}{y-1} r - \frac{y^2-y-1}{y-1} r^{-1}. 
\end{eqnarray*}
Since $S^2_{n-2}(y) + S^2_{n-3}(y) - y S_{n-2}(y) S_{n-3}(y) =1$ we get 
$$
(r - r^{-1})^2 + \left( \frac{y^2-2}{y-1} r - \frac{y^2-y-1}{y-1} r^{-1} \right)^2 - y (r - r^{-1}) \left( \frac{y^2-2}{y-1} r - \frac{y^2-y-1}{y-1} r^{-1} \right) = 1.
$$
This is equivalent to $(r^2-1)\left[ (y^3-y^2-4y+5) r^2 + y- 2\right] =0$. Since $|r| > 1$, we must have $r^2 = \frac{2-y}{y^3-y^2-4y+5}$. Note that 
 $$
1- r^2 = \frac{(y^2-3)(y-1)}{y^3-y^2-4y+5}.
 $$
The only real root of the denominator is $y \approx -2.0796$, so $y^3-y^2-4y+5 > 0$. Then $1-r^2 = (y^2-3)(y-1)/(y^3-y^2-4y+5)>0$, which contradicts $|r| > 1$. 

So we must have  $r = S_{n-2}(y) \frac{y^2-2}{y-1} - S_{n-3}(y)$. This implies that $z=(y-1)z (y+1 - \frac{S_{n-3}(y) + r}{S_{n-2}(y)} ) = (y-1)x$. Hence 
$s =0$ at $\theta = \theta_n$ (since $Q=x-xy+ (s +1)z=0$). 
\end{proof}

Since $\rho_{\pm}$ has real traces, it is conjugate to either an $\mathrm{SL}_2(\BR)$ representation or an $\mathrm{SU}(2)$ representation, see e.g. \cite[Lemma 10.1]{HP}. If the latter occurs, then $|\tr \rho_{\pm}(g)|<2$ for all $g \in G(\CK_n)$. Since $\tr \rho_{\pm}([a,w]) = x^2+y^2+z^2-xyz-2 = s+ 2 >2$ on $(\theta_n, \beta_n]$, the representation $\rho_{\pm}$ is conjugate to an irreducible $\mathrm{SL}_2(\BR)$-representation on $(\theta_n, \beta_n]$.

\begin{proposition}  $x_+ \in (-2,0)$ and $x_- \in (0, 2)$ on $[\theta_n, \beta_n)$. 
\end{proposition}

\begin{proof}
At $\theta = \theta_n$, by Proposition \ref{s} we have $s =0$. Then $Q=x-xy+ (s +1)z=0$ implies that $z = (y-1)x$, so the sign of $x_{\pm}$ is $\mp$. We have $s = x^2+y^2+z^2-xyz-4 =(y-2)(y+2-x^2)=0$. Then $x^2 = y+2 = 4 \cos^2(\theta_n/2)$, so $x_{\pm} = \mp 2 \cos(\theta_n/2)$. 

Assume that $|x|=2$ for some $[\theta_n, \beta_n)$. Since $x S_{n-2}(y) = z [(y+1) S_{n-2}(y) - S_{n-3}(y) -r]$ and  $z^2=1+r( S_{n-2}(y) - S_{n-3}(y) )$, we have  
\begin{equation} \label{eqn-x}
4 S^2_{n-2}(y) =  [ 1 + r (S_{n-2}(y) - S_{n-3}(y)) ] [(y+1)S_{n-2}(y) -S_{n-3}(y) - r ]^2.
\end{equation}

Let $t=e^{i\theta}$. Since $y = 2\cos\theta = t+t^{-1}$ we have $S_{k}(y) = (t^{k+1} - t^{-k-1}) / (t - t^{-1})$. By eliminating $r$ from $r^2 - S_{n-2}(y) r -1=0$ and equation \eqref{eqn-x} we obtain
$$(t^{2n-2}-1)(t^{2n+1}+1)(1 + t^{2 n-5}+4 t^{2 n-4}-8 t^{2 n-2}+4 t^{2 n}+t^{2 n+1}+t^{4 n-4})=0.$$

If $t^{2n-2} - 1 = 0$, then $S_{n-2}(y)=0$ which contradicts $S_{n-2}(y) \not= 0$  on $[\theta_n, \beta_n]$.

If $t^{2n+1}+1 = 0$, then $S_n(y) - S_{n-1}(y) =0$ which contradicts $S_n(y) - S_{n-1}(y)  \not= 0$  on $[\theta_n, \beta_n)$ (by Lemma \ref{Sn}). 

If the last factor equals $0$, then $t^{2n-2} + t^{2-2n} + t^3 + t^{-3} + 4 t^2 + 4 t^{-2} -8 =0$. We have
\begin{eqnarray*}
 t^3 + t^{-3} + 4 t^2 + 4 t^{-2} -8 &=& y^3-3y+4(y^2-2)-8 \\
 &=& (y+2)(y^2+2y-7) -2 \\
 &<& -2.
\end{eqnarray*}
Hence $t^{2n-2} + t^{2-2n} > 2$, contradicting $t^{2n-2} + t^{2-2n} = 2 \cos (2n-2)\theta \le 2$. 

Assume that $x=0$ for some $[\theta_n, \beta_n)$. Then $r = (y+1) S_{n-2}(y) - S_{n-3}(y)$. Since $r^2 - S_{n-2}(y) r -1=0$ and $S_{k}(y) = (t^{k+1} - t^{-k-1}) / (t - t^{-1})$, we obtain $(t^{2n-2}-1)(t^{2n+1}-1) =0$. This implies that $t^{2n+1}-1 = 0$, so $S_n(y) + S_{n-1}(y) =0$ which contradicts $S_n(y) + S_{n-1}(y)  \not= 0$ on $[\theta_n, \beta_n]$ (by Lemma \ref{Sn}).
\end{proof}

\section{Canonical longitude} \label{long}

The canonical longitude of the $(-2,3,2n+1)$-pretzel knot $\CK_n$ corresponding to the meridian $\mu : = a$ is $\lambda = a^{-(4n+7)} waw^naw^naw$. Let $\rho: G(\CK_n) \to \mathrm{SL}_2(\BC)$ be an irreducible representation. Up to conjugation we can assume that 
 $$
 \rho(\mu)= \left[
\begin{array}{cc}
M & *\\
0 & M^{-1}
\end{array}
\right] \quad \text{and} \quad \rho(\lambda) = \left[
\begin{array}{cc}
L & *\\
0 & L^{-1}
\end{array}
\right].
$$
We now recall a formula for $L$ from \cite{Ch}. Note that the longitude of $\CK_n$ considered there  is actually the inverse of $awaw^naw^naw$.  Let 
\begin{eqnarray*}
s_1 &:=& \tr \rho(bc) = \tr \rho(w) =y, \\
s_2 &:=& \tr \rho(ac) =\tr \rho(a a waw^{-1}a^{-1}) = \tr \rho(a waw^{-1}), \\
s_3 &:=& \tr \rho(ab^{-1}) = \tr \rho(a w^{-1}awaw^{-1}a^{-1}) = \tr \rho(w^{-1}awaw^{-1}), \\
\sigma &:=& s_1 + s_2 +s_3.
\end{eqnarray*} 
Note that $\tr \rho(awa^{-1}w^{-1}) = x^2+y^2+z^2-xyz-2 = s+2$. 

Since $\rho(a) + \rho(a^{-1}) = x I$ (by the Cayley-Hamilton theorem for $\mathrm{SL}_2(\BC)$) we have $s_2 + \tr \rho(awa^{-1}w^{-1}) = x^2$. Then $s_2 = x^2 - 2 - s$. Similarly,  
$s_3 + \tr \rho(w awaw^{-1}) = y s_2$ and so $s_3 = ys_2 - (xz-y)$. Hence  
\begin{eqnarray*}
\sigma  &=&  (y+1)s_2+2y-xz \\
&=& - (y+1) s -2 + (y+1)x^2 - xz.
\end{eqnarray*}

By \cite[Section 4]{Ch} we have $L = M^{-(4n+8)}\ell$ where
$$
(\ell+1)x(\sigma + 2 - \xi - x^2) = (\ell M + M^{-1}) (\sigma +2 -2x^2)
$$ for some $\xi$ satisfying $(\xi-2-s_2)s_2 = \sigma - s_2 - \xi$ and $(\sigma + 2 - 2\xi) s_3 = s_3^2-s_3 \xi + \sigma -2$. 

\no{
\begin{lemma}
$s_2 \not= -1$ for $\theta \in [\theta_n, \beta_n]$.
\end{lemma}

\begin{proof}
Assume $s_2=-1$. Then $\sigma = 0$ and $s = x^2-1$. This implies that  $xz = y-1$ and $y^2+z^2 - xyz = 3$. Hence $3 = y^2 + z^2 - y(y-1)= z^2+y$. Then $z^2 = 3 - y$, so $ r = \frac{z^2-1}{S_{n-2}(y) - S_{n-3}(y)}= \frac{2-y}{S_{n-2}(y) - S_{n-3}(y)}$. Since $r^2 - r S_{n-2}(y)-1 = 0$ we get
\begin{equation} \label{s2}
\left( \frac{2-y}{S_{n-2}(y) - S_{n-3}(y)}\right)^2 - S_{n-2}(y) \left( \frac{2-y}{S_{n-2}(y) - S_{n-3}(y)}\right) - 1 =0.
\end{equation}

Let $t=e^{i\theta}$. Since $y = 2 \cos\theta = t+t^{-1}$ we have $S_{k}(y) = (t^{k+1} - t^{-k-1}) / (t - t^{-1})$. By a direct calculation, equation \eqref{s2} becomes 
$$
(t^{2n-5} + 1)(t^{2n+1} - 2 t^{2n} - 2t + 1) =0. 
$$

If $t^{2n-5} + 1=0$ then $\theta = \frac{(2k+1)\pi}{2n-5}$. 
\end{proof}
}

Note that $\sigma +2 -2x^2 = - (y+1) s + (xy - x - z)x$. From $Q = x - x y + (s+1) z = 0$, we have $zs = xy - x - z$. Hence 
\begin{equation} \label{left}
\sigma +2 -2x^2 = - (y+1) s + xz s = (xz - y - 1)s.
\end{equation}

\begin{lemma} \label{lr}
$x(\xi - x^2)  = (z-x)s.$
\end{lemma}

\begin{proof}
If $s_2 \not= -1$, then $\xi = \frac{\sigma}{s_2+1} + s_2$. We have
\begin{eqnarray*}
\xi - x^2 &=& \frac{(y+1)s_2+2y-xz}{s_2+1} - 2 - s \\
&=& y-1-s + \frac{y-1-xz}{s_2+1}.
\end{eqnarray*}
Since $x(y-1) = (s+1)z$ we obtain
\begin{eqnarray*}
x(\xi - x^2) &=& (s+1)z - xs + \frac{(s+1)z-x^2z}{x^2-1-s} \\
&=& (s+1)z - xs - z \\
&=&  (z-x)s.
\end{eqnarray*}

If $s_2 = -1$, then $(\xi-2-s_2)s_2 = \sigma - s_2 - \xi$ becomes $\sigma=0$. Since $(\sigma + 2 - 2\xi) s_3 = s_3^2-s_3 \xi + \sigma -2$, we have $s_3 \xi = - s_3^2 + 2 s_3 + 2$. Then $s_3 \not= 0$ and $\xi = - s_3 + 2 + \frac{2}{s_3}$.

From $\sigma = (y+1)s_2+2y-xz= 0$ and $s = x^2 - 2 - s_2 = x^2-1$, we have $xz = y-1$ and $y^2+z^2 - xyz = 3$. Then $3 = (xz+1)^2 + z^2 - xz(xz+1) = xz + z^2 +1$, so $xz + z^2 = 2$. 

Since $s_3 = ys_2 - (xz-y) = - xz$, we have $\xi = xz + 2 - \frac{2}{xz}$. Hence
\begin{eqnarray*}
x(\xi - x^2) &=& x^2z + 2 x - \frac{2}{z} - x^3 \\
&=& x^2z + 2 x - (x+z) - x^3 \\
&=& (z-x)(x^2-1) \\
&=& (z-x) s.
\end{eqnarray*}
This completes the proof of the lemma.
\end{proof}

From \eqref{left} and Lemma \ref{lr} we have
\begin{eqnarray} \label{right}
x(\sigma + 2 - \xi - x^2) &=& x (\sigma +2 -2x^2) + x (x^2 - \xi) \nonumber \\
&=& (x^2z - xy - x)s + (x-z)s \nonumber \\
&=& (x^2z - xy - z)s. \label{right}
\end{eqnarray}

\no{
We have $x(\sigma + 2 - \xi - x^2) = x (\sigma +2 -2x^2) + x (x^2 - \xi) = (x^2z - xy - x)s + (x-z)s$.

From $Q = x - x y + (s+1) z = 0$, we have $s  = -(x-xy+z)/z$. Note that $\xi = \frac{\sigma}{s_2+1} + s_2$ and $s_2 = x^2 - 2 - s$.
Then by direct calculations we have
\begin{eqnarray*}
x(\sigma + 2 - \xi - x^2)  &=& (-x+xy-z)(-xy-z+x^2z)/z = - s (xy+z-x^2z), \\
\sigma + 2 - 2 x^2 &=& (-x+xy-z)(-1-y+xz)/z = - s (1+y-xz).
\end{eqnarray*}
}

Let $C=x^2z - xy - z$ and $D=xz - y - 1$. From \eqref{left} and  \eqref{right} we obtain
$$
(\ell+1)s C= (\ell M + M^{-1}) s D.
$$

We now focus on the representation $\rho_{\pm} (\theta): G(\CK_n) \to \mathrm{SL}_2(\BC)$, for $\theta \in [\theta_n, \beta_n]$, defined in Section \ref{path}. Then we have the following. 


\begin{proposition} \label{4}
 If $n \not= 4$ then $D^2-C^2 > 0$ for $\theta \in [\theta_n, \beta_n)$.
\end{proposition}

\begin{proof}
The proof is divided into two steps.

\underline{Step 1}: We first claim that $D^2-C^2 > 0$ at $\theta =  \theta_n$. Indeed, since $\theta =  \theta_n$ we have $s = 0$, so $Q=x-xy+ (s +1)z=0$  implies that $z = (y-1)x$. Then $s=(y-2)(y+2-x^2)$. So $y=x^2-2$ and $z=x^3-3x$. By a direct calculation, we have 
\begin{eqnarray*}
D^2-C^2 &=& - x^{10} + 11 x^8 - 43 x^6 +  68 x^4 - 33 x^2 + 1 \\
&=& - y^5 + y^4 + 5 y^3 - 6 y^2 - 5 y + 7.
\end{eqnarray*}
The real roots of the above polynomials are approximately $-1.96757$, $-1.22062$ and  $1.66787$. Note that $y < 0$, since $\theta_n \in ( \pi/2, \pi)$. 

If $n \not\equiv 1 \pmod{3}$, then by Proposition \ref{root} we have $\theta_n < \frac{2\pi}{3}$. So $ y > 2 \cos \frac{2\pi}{3}= -1$. 

If $n \equiv 1 \pmod{3}$ then  $\theta_n \in (  \frac{2\pi}{3} + \frac{\pi}{6n+6}, \frac{2\pi}{3} +  \frac{4\pi}{12n+9} )$. If $n \ge 10$, then $y > 2 \cos (\frac{2\pi}{3} + \frac{4\pi}{129}) \approx -1.16372$. If $n=7$, then by a direct calculation we have $\theta_n \approx 2.20391$ and so $y \approx -1.18332$. 

Hence, at $\theta =  \theta_n$ we have $y \in (-1.22062, 1.66787)$ and so $D^2-C^2 > 0$.

\underline{Step 2}:  We now prove that $D^2-C^2 \not= 0$ on $[\theta_n, \beta_n)$. Assume that $D^2-C^2=0$ for some $\theta \in [\theta_n, \beta_n)$. Then  
\begin{equation} \label{CD}
(-1-y+xy+z+xz-x^2z)(1+y+xy+z-xz-x^2z)=0.
\end{equation}

If $x=1$ then $(-1+z)(1+2y-z)=0$. Since $|z| > 1$, we must have $z=2y+1$. The equation $Q = x-xy+(s +1)z =0$ implies that $3y^2(2y+3)=0$, which contradicts $0 > y > 2\cos \beta_n \ge  2 \cos (\frac{2\pi}{3} + \frac{\pi}{15}) \approx -1.33826$. 


Similarly, the case $x=-1$ cannot occur. Hence $x \not= \pm 1$. Then \eqref{CD} becomes $y = \frac{1-z-xz+x^2z}{x-1}$ or $y = \frac{-1-z+xz+x^2z}{x+1}$. From $x-xy+(s +1)z =0$ we have $(x-2)(x-z)(-1+x-z+z^2)=0$ or $(x+2)(x-z)(1+x-z-z^2)=0$ respectively.  

Note that $x \in (-2,2)$. If $x=z$ then $y+1-\frac{S_{n-3}(y)+r}{S_{n-2}(y)} = \frac{x}{z} = 1$, so $r = y S_{n-2}(y) - S_{n-3}(y) = S_{n-1}(y)$. Since $r - 1/r=S_{n-2}(y)$ we get 
$$
S_{n-2}(y)[ (y-1)S_{n-1}(y) - S_{n-2}(y)]=0.
$$ 
The last factor is equal to $S_{n}(y) - S_{n-1}(y)$, which is non-zero on $[\theta_n, \beta_n)$. Hence $x \not= z$. 

If $y = \frac{1-z-xz+x^2z}{x-1}$ then $-1+x-z+z^2=0$. This implies that $x=1+z-z^2$ and so $y = -(z+1)(z^3-2z^2+z-1)/z = z^{-1} + z + z^2 - z^3$. From $x= z (1 + \frac{S_{n-1}(y) - r}{S_{n-2}(y)})$ we have 
\begin{eqnarray*}
r &=& S_{n-1}(y) + (1-x/z) S_{n-2}(y) \\
&=& S_{n-1}(y) + (z-1/z) S_{n-2}(y). 
\end{eqnarray*}
Since $r - r^{-1}=S_{n-2}(y)$, by a direct calculation we obtain
$$
(z^4-z^3-3z^2+z+1) S_{n-2}(y) (z^{-1} S_{n-2}(y) - S_{n-1}(y)) =0.
$$ 
If $z^4-z^3-3z^2+z+1=0$ then $y^2-5=0$, which contradicts $y =2\cos\theta$. Hence $z^{-1} S_{n-2}(y) - S_{n-1}(y) = 0$. Since $S_{n-2}(y) \not= 0$ we have $z  = \frac{S_{n-2}(y)} {S_{n-1}(y)}$. Then $r = (z-1/z) {S_{n-2}(y)} + {S_{n-1}(y)} = {S^2_{n-2}(y)} / {S_{n-1}(y)}$. Combining with $r - 1/r=S_{n-2}(y)$ we obtain 
\begin{equation} \label{F}
S^2_{n-1}(y) + ( S_{n-1}(y) - S_{n-2}(y) ) S^3_{n-2}(y) =0.
\end{equation} 

Similarly, if $y = \frac{-1-z+xz+x^2z}{x+1}$ then $1+x-z-z^2=0$  and \eqref{F} also holds true.  This contradicts Lemma \ref{lem}. Hence $D^2-C^2 \not= 0$ on $[\theta_n, \beta_n)$. Since it is positive at $\theta = \theta_n$, we must have $D^2-C^2 > 0$ for $\theta \in [\theta_n, \beta_n)$.
\end{proof}

\begin{remark} \label{Rem}
Proposition \ref{4} does not hold true for $n=4$. Indeed, if $n=4$ then $\theta_n \approx 2.2728$ and $\beta_n \approx 2.44346$. At $\theta = \theta_n$, we have $y \approx -1.2915$ and so $D^2-C^2 = - y^5 + y^4 + 5 y^3 - 6 y^2 - 5 y + 7 < 0$. However, at $\theta = 2.3 \in (\theta_n, \beta_n)$ we have $D^2-C^2 \approx 1.13479 > 0$. Hence the sign of $D^2-C^2$ changes on $[\theta_n, \beta_n)$.
\end{remark}

Note that $s > 0$ on $(\theta_n, \beta_n]$ (by Proposition \ref{s}). Hence 
$$
L = M^{-(4n+8)}\ell = M^{-(4n+9)} \frac{CM - D}{DM - C}
$$ 
holds true for all $\theta \in (\theta_n, \beta_n)$. Although $s|_{\theta = \theta_n} = 0$ (reducible representation so $L=1$), the above formula for $L$ still holds true at $\theta = \theta_n$. In fact, at $\theta = \theta_n$ we have $y=x^2-2 = M^2+M^{-2}$ and $z=x^3-3x = M^3 + M^{-3}$. Note that  we can take $M = e^{i\theta_n/2}$, since $y = 2\cos\theta_n$. By a direct calculation, $M^{-(4n+9)} \frac{CM - D}{DM - C} = 1$ becomes $M^{4n+10}-M^{4n+6}+M^{4n+4}+M^6-M^4+1=0$. This is equivalent to $\Delta_{\CK_n}(M^2)=0$, which holds true since $\Delta_{\CK_n}(e^{i\theta_n}) = 0$.

At $\theta = \beta_n$,  we have $x_{\pm} =  \mp 2$ and $z_{\pm} = \pm (1-y)$. Then $C \pm D = 0$, so $M^{-(4n+9)} \frac{CM - D}{DM - C}$ is an indeterminate form $\frac{0}{0}$.  However, we will show the following. 

\begin{proposition} \label{L}
At $\theta = \beta_n$ we have $L = 1$. 
\end{proposition}

\begin{proof}
We compute the trace of $\rho(awaw^naw^naw)$ by applying Lemma \ref{power}. First, we have 
\begin{eqnarray*}
\rho(awaw^naw^naw) &=& S^2_{n-1}(y) \rho(awawawaw) + S^2_{n-2}(y) \rho(awa^3w) \\
&& - \, S_{n-1}(y) S_{n-2}(y) (\rho(awawa^2w) + \rho(awa^2waw) )
\end{eqnarray*}
Note that $\tr \rho ((aw)^k) = S_{k-1}(z) \tr \rho(aw) - S_{k-2}(z) \tr I= z S_{k-1}(z)  - 2 S_{k-2}(z)$. Hence
\begin{eqnarray*}
\tr \rho ((aw)^2) &=& z^2-2, \\
\tr \rho ((aw)^3) &=& z^3 - 3z, \\
\tr \rho ((aw)^4) &=& z^4 - 4z^2 + 2.
\end{eqnarray*}

By Lemma \ref{power} we have 
\begin{eqnarray*}
\tr \rho(awa^3w) &=& S_2(x) \tr \rho(awaw) - S_1(x) \tr \rho(aw^2) \\
&=& (x^2-1) (z^2-2) - x(y z - x),
\end{eqnarray*}
and 
\begin{eqnarray*}
\tr \rho(awawa^2w) &=& x \tr \rho(awawaw) - \tr \rho(awaw^2) \\
&=& x (z^3 - 3z) - (y \tr \rho(awaw)  - \rho(awa)) \\
&=& x (z^3 - 3z) - (y (z^2-2)  - (xz-y)).
\end{eqnarray*}
Similarly, $\tr  \rho(awa^2waw) = x (z^3 - 3z) - (y (z^2-2)  - (xz-y))$. Hence
\begin{eqnarray*}
\tr \rho(awaw^naw^naw) &=& S^2_{n-1}(y) (z^4 - 4z^2 + 2) + S^2_{n-2}(y) [(x^2-1) (z^2-2) - x(y z - x)]\\
&& - \, 2 S_{n-1}(y) S_{n-2}(y) [x (z^3 - 3z) - (y (z^2-2)  - (xz-y))].
\end{eqnarray*}

At $\theta= \beta_n$, by Proposition \ref{x}  we have $x_{\pm} = \mp 2$ and $z_{\pm} = \pm (1-y)$. Moreover, $S^2_{n-2}(y) = \frac{(y-1)^2}{2-y}$, $S^2_{n-1}(y) = \frac{1}{2-y}$ and $S_{n-1}(y) S_{n-2}(y) = \frac{1}{y-1} S^2_{n-2}(y) = \frac{y-1}{2-y}$. Then by a direct calculation we have $\tr \rho(awaw^naw^naw) = 2$. Hence $L = 1$ at $\theta = \beta_n$. 
\end{proof}

\section{Left-orderable surgeries} \label{LO}

In this section we will use continuous paths of irreducible $\mathrm{SL}_2(\BR)$-representations constructed in Section \ref{path} to prove Theorem \ref{main}. 

\subsection{Elliptic $\bm{\mathrm{SL}_2(\BR)}$-representations} 

An element of $\mathrm{SL}_2(\BR)$ is elliptic if its trace is a real number in $(-2,2)$. A representation $\rho: \BZ^2 \to \mathrm{SL}_2(\BR)$ is called elliptic if the image group $\rho(\BZ^2)$ contains an elliptic element of $\mathrm{SL}_2(\BR)$. In which case, since $\BZ^2$ is an abelian group, every non-trivial element of $\rho(\BZ^2)$ must also be elliptic.

Recall from the previous section that for the representation $\rho_{\pm} (\theta): G(\CK_n) \to \mathrm{SL}_2(\BC)$, with $\theta \in [\theta_n, \beta_n)$, we have $L = M^{-(4n+9)} \frac{CM - D}{DM - C}$. Let $\phi_{\pm} = \arccos (x_{\pm}/2) \in (0, \pi)$. Then we can take $M = e^{i\phi_{\pm}}$. Note that $|L |= 1$. We have 
\begin{eqnarray*}
\mathrm{Re} \, \frac{CM - D}{DM - C} &=& \frac{2CD - (C^2+D^2)\cos\phi_{\pm} } {C^2+D^2 - 2CD \cos\phi_{\pm} }, \\
\mathrm{Im} \, \frac{CM - D}{DM - C} &=& \frac{(D^2 - C^2) \sin \phi_{\pm}} {C^2+D^2 - 2CD \cos\phi_{\pm} } > 0,
\end{eqnarray*}
since $D^2-C^2 > 0$ for $\theta \in [\theta_n, \beta_n)$. 
Hence we can write $L = e^{i \varphi_{\pm}}$ where
\begin{eqnarray*}
\varphi_{-} &=& 2 \left\lfloor \frac{2n+4}{3} \right\rfloor\pi  - (4n+9) \phi_{-} + \arccos \frac{2CD - (C^2+D^2)\cos\phi_{-}}{C^2+D^2 - 2CD \cos\phi_{-}}, \\
\varphi_{\pm} &=& (4n+8)\pi - 2 \left\lfloor \frac{2n+4}{3} \right\rfloor\pi  - (4n+9) \phi_{+} + \arccos \frac{2CD - (C^2+D^2)\cos\phi_{+}}{C^2+D^2 - 2CD \cos\phi_{+}}.
\end{eqnarray*}

\begin{lemma} \label{angle}
The following hold true.

$(1)$ $\varphi_{\pm}(\theta_n)=0$.

$(2)$ As $\theta \to \beta_n$, we have $\phi_+ \to \pi$, $\phi_- \to 0$ and $\varphi_{\pm} \to \mp 2 \left\lfloor \frac{2n+4}{3} \right\rfloor\pi$. 
\end{lemma}

\begin{proof}
$(1)$ As $\theta = \theta_n$ we have $x_{\pm} = \mp 2 \cos(\theta_n/2)$, so $\phi_- = \theta_n/2$ and $\phi_+ = \pi - \theta_n/2$. Note that $L(\theta_n)=1$, so $\varphi_{\pm}(\theta_n) \in 2\pi \BZ$. 

Suppose $n \equiv 0 \pmod{3}$. Since $\theta_n \in (\alpha_n, \gamma_n) = (\frac{2\pi}{3}-\frac{\pi}{6n+15}, \frac{2\pi}{3})$ we have 
$$
\frac{(4n+6)\pi}{3}   - (4n+9) \frac{\pi}{3}    < \varphi_-(\theta_n) < \frac{(4n+6)\pi}{3}  - (4n+9)(\frac{\pi}{3} -\frac{\pi}{12n+30}) + \pi,
$$
i.e.
$$ - \pi  < \varphi_-(\theta_n) < \frac{(4n+9)\pi}{12n+30}.
$$
Similarly, we also have $ - \frac{(4n+9)\pi}{12n+30}  < \varphi_+(\theta_n) <  \pi$. 

Since $\varphi_{\pm}(\theta_n) \in 2\pi \BZ$, we obtain $\varphi_{\pm}(\theta_n)=0$. The cases $n \equiv \pm 1 \pmod{3}$ are similar. 

$(2)$ As $\theta \to \beta_n$, we have $x_{\pm} \to \mp 2$ and $L \to 1$ (by Proposition \ref{L}). Then $\phi_+ \to \pi$, $\phi_- \to 0$ and $\frac{CM - D}{DM - C} \to \mp 1$. This implies that $\arccos \frac{2CD - (C^2+D^2)\cos\phi_{+}}{C^2+D^2 - 2CD \cos\phi_{+}} \to \pi$ and $ \arccos \frac{2CD - (C^2+D^2)\cos\phi_{-}}{C^2+D^2 - 2CD \cos\phi_{-}} \to 0$.  Hence $\varphi_{\pm} \to \mp 2 \left\lfloor \frac{2n+4}{3} \right\rfloor\pi$.
\end{proof}

Let $X$ be the complement of an open tubular neighborhood of the $(-2,3,2n+1)$-pretzel knot $\CK_n$ in $S^3$, where $n \ge 3$ is an integer and $n \not= 4$. 

Let $X_{\frac{m}{l}}$ be the 3-manifold obtained from $S^3$ by $\frac{m}{l}$-surgery on $\CK_n$.

\begin{proposition} \label{prop} 
For each rational number $\frac{m}{l}\in(-\infty, 0) \cup (0,2 \left\lfloor \frac{2n+4}{3} \right\rfloor)$ there exists a  representation $\rho: \pi_1(X_{\frac{m}{l}}) \to \mathrm{SL}_2(\BR)$ such that $\rho|_{\pi_1(\partial X)}: \pi_1(\partial X) \cong \BZ^2 \to \mathrm{SL}_2(\BR)$ is an elliptic representation. 
\end{proposition}

\begin{proof}
First, Lemma \ref{angle} implies that the image of the continuous function $- \frac{\varphi_-}{\phi_-}$ on $(\theta_n, \beta_n)$ contains $(-\infty,0)$ and that of $- \frac{\varphi_+}{\phi_+}$  contains $\left( 0, 2 \left\lfloor \frac{2n+4}{3} \right\rfloor \right)$. 

Suppose $\frac{m}{l}\in  (-\infty,0) \cap \BQ$. Then $\frac{m}{l}= - \frac{\varphi_-(\theta)}{\phi_-(\theta)}$ for some $\theta \in (\theta_n, \beta_n)$. Since $M^m L^l = e^{i(m\theta_- + l\varphi_-)} = 1$, we have $\rho_-(\mu^m \lambda^l) = I$. Let $\rho: G(\CK_n) \to \mathrm{SL}_2(\BR)$ be an $\mathrm{SL}_2(\BR)$ representation conjugate to $\rho_-$. Then we also have $\rho(\mu^m \lambda^l) = I$. This means that $\rho$ extends to a  representation $\rho: \pi_1(X_{\frac{m}{l}}) \to \mathrm{SL}_2(\BR)$. Note that $\rho|_{\pi_1(\partial X)}$ is an elliptic representation. 

The case $\frac{m}{l}\in  (0,2 \left\lfloor \frac{2n+4}{3} \right\rfloor) \cap \BQ$ is similar, since $\frac{m}{l}= - \frac{\varphi_+(\theta)}{\phi_+(\theta)}$ for some $\theta \in (\theta_n, \beta_n)$. 
\end{proof}

\subsection{Proof of Theorem \ref{main}} Suppose  $\frac{m}{l} \in (-\infty, 2 \left\lfloor \frac{2n+4}{3} \right\rfloor) \cap \BQ$. 

If $\frac{m}{l} = 0$, then $X_{\frac{m}{l}}$ is a prime 3-manifold whose first Betti number is $1$, and by \cite{BRW} such manifold has left-orderable fundamental group.

If $\frac{m}{l} \not= 0$, then by Proposition \ref{prop} there exists a  representation $\rho: \pi_1(X_{\frac{m}{l}}) \to \mathrm{SL}_2(\BR)$ such that $\rho|_{\pi_1(\partial X)}$ is elliptic. Then $\rho$ lifts to a representation $\widetilde{\rho}: \pi_1(X_{\frac{m}{l}}) \to \widetilde{\mathrm{SL}_2(\BR)}$, where $\widetilde{\mathrm{SL}_2(\BR)}$ is the universal covering group of $\mathrm{SL}_2(\BR)$. See e.g. \cite[Section 3.5]{CD}. Note that $X_{\frac{m}{l}}$ is an irreducible 3-manifold (by \cite{Wu}) and $\widetilde{\mathrm{SL}_2(\BR)}$ is a left-orderable group (by \cite{Be}). Hence, by \cite{BRW}, the group $\pi_1(X_{\frac{m}{l}})$ is  left-orderable.

\no{
Suppose $n = 4$. Then $\theta_n \approx 2.2728$ so $y \in (2\cos(\frac{2\pi}{3} + \frac{\pi}{9}), 2\cos\theta_n) \approx (-1.53209, -1.2915)$.  At $\theta = \theta_n$ we have $y = -1.2915$ and 
$$
D^2-C^2 = - y^5 + y^4 + 5 y^3 - 6 y^2 - 5 y + 7.
$$
The real roots are $-1.96757, \, -1.22062$ and  $1.66787$. Hence  $D^2-C^2 < 0$. 

Assume that $D^2-C^2=0$ for some $\theta \in [\theta_n, \beta_n)$. Then 
$$
P(\theta) := S^2_{n-1}(y) + ( S_{n-1}(y) - S_{n-2}(y) ) S^3_{n-2}(y) = 0.
$$
Note that $P(y) = -1 + 2 y + 8 y^2 - 7 y^3 - 10 y^4 + 9 y^5 + 5 y^6 - 5 y^7 - y^8 + y^9$ has only one real root $y = -1.31202$ in the interval $(-1.53209, -1.2915)$, which corresponds to $c \approx 2.28632$. At this root, we have $z  = \frac{S_{n-2}(y)} {S_{n-1}(y)} \approx 1.97361$ and  $x=1+z-z^2 \approx -0.921516$. Then $C=xy+z-x^2z \approx 1.50669$ and $D=1+y-xz \approx 1.50669$. 
}

\section*{Acknowledgements} 
The author has been supported by a grant from the Simons Foundation (\#708778). He would like to thank the referee for helpful comments and suggestions.

\end{document}